\documentclass[a4paper,11pt]{amsart}

\usepackage{amssymb}
\usepackage{epsfig}
\usepackage{amsfonts}
\usepackage{amsmath}
\usepackage{euscript}
\usepackage{amscd}
\usepackage{amsthm}
\DeclareMathAlphabet{\mathpzc}{OT1}{pzc}{m}{it}
\usepackage[pagebackref,hypertexnames=false, colorlinks, citecolor=red,linkcolor=blue, urlcolor=red]{hyperref}
\usepackage{mathrsfs}

\DeclareSymbolFont{SY}{U}{psy}{m}{n}
\DeclareMathSymbol{\emptyset}{\mathord}{SY}{'306}

\makeatletter
\@namedef{subjclassname@2020}{%
  \textup{2020} Mathematics Subject Classification}
\makeatother

\theoremstyle{plain}

\newtheorem{thm}{Theorem}[section]
\newtheorem{cor}[thm]{Corollary}
\newtheorem{lem}[thm]{Lemma}
\newtheorem{prop}[thm]{Proposition}
\theoremstyle{definition}
\newtheorem{defn}[thm]{Definition}
\newtheorem{rem}[thm]{Remark}

\numberwithin{equation}{section}
\def\A{{\mathcal A}}
\def\C{{\mathbb C}}
\def\D{{\mathbb D}}
\def\E{{\mathcal E}}
\def\O{{\mathcal O}}
\def\H{{\mathcal H}}
\def\HM{{\mathscr M}}
\def\K{{\mathscr K}}
\def\B{{\mathcal B}}
\def\dbar{\bar\partial}
\def\norm#1{\left\|{#1}\right\|}

\def\inp#1,#2{\left\langle{#1},{#2}\right\rangle}
\def\N{\mathbb{N}}

\def\J{\mathcal J}

\def\s{\psi}
\def\si{\sigma}

\def\l{\lambda}
\def\o{\omega}
\def\O{\Omega}
\def\f{\textbf{\textit{s}}}

\def\ra{\rightarrow}
\def\ov{\overline}

\def\a{\alpha}
\def\b{\beta}
\def\g{\gamma}
\def\G{\Gamma}
\def\e{\varepsilon}
\def\r{\tilde{\rho}}
\def\d{\delta}

\def\h{ hermitian holomorphic vector bundle}

\def\w{with respect to }

\def\beq{\begin{eqnarray}}
\def\eeq{\end{eqnarray}}
\def\beqa{\begin{eqnarray*}}
\def\eeqa{\end{eqnarray*}}

\def\del{\partial}
\def\T{\boldsymbol{T}}
\def\M{\boldsymbol{M}}
\def\bz{\boldsymbol{z}}
\def\bw{\boldsymbol{w}}
\def\rkhs{reproducing kernel Hilbert space }
\def\<{\langle}
\def\>{\rangle}
\def\Z{\mathcal{Z}}

\def\z{\zeta}
\def\nbhd{neighbourhood}

\newcommand{\be}{\begin{equation}}
\newcommand{\ee}{\end{equation}}
\newcommand{\bea}{\begin{eqnarray}}
\newcommand{\eea}{\end{eqnarray}}
\newcommand{\Bea}{\begin{eqnarray*}}
\newcommand{\Eea}{\end{eqnarray*}}

\newcounter{cnt1}
\newcounter{cnt2}
\newcounter{cnt3}
\newcommand{\blr}{\begin{list}{$($\roman{cnt1}$)$}
 {\usecounter{cnt1} \setlength{\topsep}{0pt}
 \setlength{\itemsep}{0pt}}}
\newcommand{\bla}{\begin{list}{$($\alph{cnt2}$)$}
 {\usecounter{cnt2} \setlength{\topsep}{0pt}
 \setlength{\itemsep}{0pt}}}
\newcommand{\bln}{\begin{list}{$($\arabic{cnt3}$)$}
 {\usecounter{cnt3} \setlength{\topsep}{0pt}
 \setlength{\itemsep}{0pt}}}
\newcommand{\el}{\end{list}}

\begin{document}

\title [ON UNITARY INVARIANTS OF QUOTIENT HILBERT MODULES]{ON UNITARY INVARIANTS OF QUOTIENT HILBERT MODULES ALONG SMOOTH COMPLEX ANALYTIC SETS}
\author{Prahllad Deb}
\address{Department of Mathematics, Indian Institute of Science, Bangalore - 560012, Karnataka, India}
\email{prahllad.deb@gmail.com}

\thanks{This work is supported by Senior Research Fellowship funded by Indian Institute of Science Education and Research Kolkata (IISER Kolkata). Most of the results in this paper are from the PhD thesis of the author submitted to the IISER Kolkata.}



\begin{abstract}

Let $\O \subset \C^m$ be an open, connected and bounded set and $\A(\O)$ be a function algebra of holomorphic functions on $\O$. Suppose that $\HM$ is a reproducing kernel Hilbert module over $\A(\O)$. In this article, we first obtain a model for the quotient Hilbert modules obtained from submodules of functions in $\HM$ vanishing to order $k$ along a smooth irreducible complex analytic set $\Z\subset\O$ of codimension at least $2$, assuming that the Hilbert module $\HM$ is in the Cowen-Douglas class over $\O$. This model is used to show that such a quotient module happens to be in the Cowen-Douglas class over $\Z\cap\O$ which then enables us to determine unitary equivalence classes of the aforementioned quotient modules in terms of the geometric invariants of hermitian holomorphic vector bundles. As an application, we obtain that unitary equivalence classes of a large family of these Hilbert modules are completely determined by those of certain quotient modules of the above kind.
\end{abstract}

\maketitle
\section{Introduction}\label{intro}
Let $\O$ be a bounded domain in $\C^m$ and $\A(\O)$ be the unital Banach algebra obtained as the norm closure \w the supremum norm on $\ov{\O}$ of all functions holomorphic on a neighbourhood of $\ov{\O}$. A complex Hilbert space $\H$ is said to be a Hilbert module over $\A(\O)$ with module map $\A(\O)\times\H \overset{\pi}\ra \H$ by point-wise multiplication such that the module action $\A(\O)\times \H \overset{\pi}\ra \H$ is norm continuous. Classical examples of Hilbert modules are the Hardy and Bergman module over the disc algebra $\A(\D)$ where $\D$ is the open unit disc in the complex plane. Two Hilbert modules $\H_1$ and $\H_2$ over $\A(\O)$ with module actions $(f,h_i)\mapsto M^{(i)}_{f}(h_i) $, $i=1,2$, are said to be isomorphic if there is a Hilbert space isomorphism $\Phi: \H_1 \ra \H_2$ such that $\Phi(M^{(1)}_{f}(h_1))=M^{(2)}_{f}(\Phi(h_1))$. The basic problem alluded to the title is as follows:

\textit{Given a Hilbert module $\HM$ and a submodule $\HM_0$ over $\A(\O)$, satisfying the exact sequence $$0 \ra \HM_0 \overset{i}{\ra} \HM \overset{\pi}{\ra} \HM_q \ra 0,$$ where $i$ is the inclusion map, $\pi$ is the quotient map and $\HM_q$ is the quotient module $\HM\ominus\HM_0$, is it possible to determine $\HM_q$ in terms of $\HM$ and $\HM_0$? One can make this general question more precise by asking if it is possible to assign some computable invariants on $\HM_q$ in terms of $\HM$ and the submodule $\HM_0$.}

In \cite{GIRHM} and \cite{GIRQMGRID}, these questions have been studied when the submodule $\HM_0$ consists of all functions in a quasi-free Hilbert module $\HM$ of rank $1$ (cf. \cite[Section 2]{QFRHM}, \cite[Page 3]{OQFHM}) over $\A(\O)$ vanishing along a smooth hypersurface $\Z$ and a complex algebraic variety of complete intersection of finitely many smooth hypersurfaces in $\O$, respectively, to obtain geometric invariants for such quotient modules in terms of the curvature of the line bundle associated to $\HM$ and $\HM_0$. By means of jet construction, \cite[Page 372]{OQMAM} a model for the quotient modules obtained from the submodules $\HM_0^{(k)}$ of functions in $\HM$ vanishing of order $k\geq 1$ along $\Z$ has been provided which is then used to obtain geometric invariants of such quotient modules generalizing the results in \cite{GIRHM}. Later in \cite{EQHMII}, a complete set of unitary invariants for such quotient modules with an arbitrary $k$ has been described. For $k=2$, these invariants turn out to be the tangential and transverse components of the curvature of the line bundle $E_{\HM}$ relative to hypersurface $\Z$ and the second fundamental form for the inclusion $E_{\HM}\subset J^{(2)}_1E_{\HM}$ where $J^{(2)}_1E_{\HM}$ is the second order jet bundle of $E_{\HM}$ relative to $\Z$ (cf. \cite[Section 3]{EQHMII}). More recently, Chen and Douglas have introduced the localization of a commuting tuple of operators in the Cowen-Douglas class over $\O$ on $\Z$ and obtained unitary invariants for them in \cite{ALOTCD}. Furthermore, they relate these localizations to the quotient modules obtained from the submodules of vector valued holomorphic functions on $\O \subset \C^m$ vanishing to order $k\geq 2$ on a smooth hypersurface $\Z$. These studies on quotient modules naturally lead to consider the case when the submodules consist of vector valued holomorphic functions vanishing to higher order along a smooth complex analytic set of codimension greater than $1$.  We must first recall the definition of the Cowen-Douglas class. 

Let $D_{\T} : \H \to \H \oplus \cdots \oplus \H$ be the operator given by the formula: $D_{\T} h = \left(T_1 h,\hdots,T_m h\right)$, $h \in \H.$
Following \cite{OPOSE}, we say that a commuting $m$-tuple of bounded linear operators $\T=(T_1,\hdots,T_m)$ on $\H$ is in $\mathrm B_r(\O)$ if 
\begin{itemize} 
\item $\dim \ker D_{\T - z I}=r$, $z\in \O$;
\item  $\text{ran} D_{\T - z I}$ is closed in $\H\oplus\cdots\oplus\H$;
\item the linear span of the vectors in $\ker D_{\T - z I}$, $z\in \O$ is dense in $\H$. 
\end{itemize}
Following the ideas of \cite{OPOSE}, it is easy to establish a one to one correspondence between the unitary equivalence class of commuting $m$-tuples of operators in $\mathrm B_r(\O)$ and equivalence class of  the corresponding {\h} $E_{\T}:=\{(w, x)\in\Omega\times\mathcal H: x\in \ker{D}_{{\T}-w}\}$ of rank $n$ over $\O$. The equivalence class of the {\h} is the local equivalence of the hermitian structure. These vector bundles are distinguished, among others, by the property that the hermitian structure on the fibre over $z\in \O$ is induced from the inner product of a fixed Hilbert space $\H$. It has been proved in \cite{GBKCD} that the corresponding $m$-tuple of operators $\T$ is simultaneously unitarily equivalent to the adjoint of the $m$-tuple of multiplication operators $\M=(M_{z_1},\hdots,M_{z_m})$ by the coordinate functions on a \rkhs $\H_K$ consisting of $\C^r$- valued holomorphic functions on $\O^*:=\{\ov{z}:z\in\O\}$, where $K$ is the reproducing kernel on $\O^*$.

In general, the adjoint of the tuple $\M$ need not be in ${\mathrm B}_r(\Omega)$.  However, it can be ensured by putting additional conditions (cf. \cite{GBKCD}) on $K$. One set of such conditions are the following:
\begin{itemize}
\item[(i)] for any $z\in\O^*$, the evaluation mapping $ev_z:\H_K\ra\C^r$ is bounded and surjective;
\item[(ii)] $\H_K$ has the \textit{Gleason property}, that is, for any $z\in\O^*$ and $f\in\H_K$, $f(z)=0$ if and only if $f(w)=(w_1-z_1)f_1+\cdots+(w_m-z_m)f_m$ for some $f_1,\hdots,f_m\in\H_K$ and any $w\in\O^*$.
\end{itemize}

The following result shows that these are, indeed, the necessary and sufficient for a tuple $\T=(T_1,\hdots,T_m)$ of commuting operators to be in $\mathrm B_r(\O)$.

\begin{thm}\cite[Theorem 2]{ALOTCD}\label{gleason}
A tuple of operators $\T=(T_1,\hdots,T_m)$ lies in $\mathrm B_r(\O)$ if and only if $\T$ is unitarily equivalent to the adjoint of the tuple $\M=(M_{z_1},\hdots,M_{z_m})$ of multiplication operators on a $\C^r$- valued holomorphic function space over $\O^*$ satisfying (i) and (ii) as listed above. 
\end{thm}
From now on, we say that a Hibert module $\HM$ is in $\mathrm B_r(\O)$ if $(M^*_{z_1},\hdots,M^*_{z_m})$ is in $\mathrm B_r(\O)$ where $(M_{z_1},\hdots,M_{z_m})$ is the tuple of multiplication operators on $\HM$ corresponding to the coordinate functions on $\O^*$.

In the first half of the present article, the emphasis on treating the quotient modules more generally, compared to previous work in special cases (codimension $1$, scalar valued functions, etc.), in a uniform way. However, as in the previous case, the basic assumption made throughout is that the subvariety $\Z$ is smooth. Thus after a local change of variables at a given point in $\O$, $\Z$ becomes locally isomorphic to a linear subvariety (codimension $d\geq 2$) described by the vanishing of first $d$ variables yielding geometric invariants of the quotient modules \w the local coordinates in a neighbourhood of a given point (by abstract principles this often suffices to obtain global invariants). We make use of this description of the subvariety $\Z$ in Section \ref{SM} to obtain a canonical model for the submodule of functions in $\HM$ vanishing to order $k$ on $\Z$. The main crux in obtaining this model is to prove the Proposition \ref{coch} which can also be obtained using the multi-variable version of the Fa di Bruno formula. However, the advantage of the proof given in this article is in providing the change of variable matrix for the $N$-jet of a holomorphic function explicitly which we use in later sections. 

This model of submodules then enables us to obtain a model for the quotient modules (Theorem \ref{qm}) in Section \ref{QM}. The main idea, originating from \cite{OQMAM}, is to consider a module $J(\HM)$ of $k$-jets relative to $\Z$ of all functions in $\HM$ and identify the quotient module $\HM_q$ as the restriction of $J(\HM)$ to $\Z$. One of the basic problems is whether this new module $J(\HM)|_{\Z}$ over $\A(\Z)$ belongs to $\mathrm B_{r|A|}(\Z)$ where $A=\{\a\in(\N\cup\{0\})^m:|\a|<k\}$. As one of the main results in Section \ref{QM}, it is shown that $$\HM\in\mathrm B_r(\O)\Longrightarrow J(\HM)|_{\Z}\in\mathrm B_{r|A|}(\Z)$$ at least for a natural class of Hilbert modules in the Cowen-Douglas class in Theorem \ref{qmodinCDcls}. 

Thus for this class of Hilbert modules, the corresponding quotient modules $\HM_q$ give rise to a {\h}s over $\Z$ as well, but it turns out that only the bundle structure cannot determine the unitary equivalence class of such quotient modules. The flag structure of this new bundle induced by the module action $-$ which now involves the nilpotent action obtained from the compression of the tuple of multiplication operators by the vanishing coordinates of $\Z~-$ must be accounted as explained in Remark \ref{rem} and Remark \ref{rjbi}. Classifying such quotient modules $-$ for vector valued functions and the higher codimensional vanishing set $-$ requires new ideas and techniques extending older ones from complex geometry of jet bundles and moving frames.

The vector bundle associated with the module $J(\HM)|_{\Z}$ turns out to be the $k$-th order jet bundle relative to $\Z$ of the vector bundle associated with the module $\HM$. In Section \ref{JB}, using the techniques of normalised frame \cite[Lemma 2.3]{DACM} of these vector bundles, we first show in Theorem $\ref{jean1}$ that two such quotient modules are unitarily equivalent if and only if there exists a constant isometric jet bundle isomorphism between the corresponding jet bundles restricted to the submanifold $\Z$ with the aid of normalized frame. This fact is then used to determine the unitary invariants of the aforementioned quotient modules in Theorem \ref{mthm}. These results extend the results in \cite{EQHMI,EQHMII,ALOTCD} to the case of quotient modules obtained from submodules of vector valued functions vanishing along a smooth analytic variety of arbitrary codimension. 

Finally, in Theorem \ref{localization} we show, for two Hilbert modules $\HM,\tilde{\HM}\in\mathrm B_r(\O)$ with $r<m$ and any connected complex submanifold $\Z\subset \O$ of codimension $d$ with $r\leq d\leq m$, that the unitary equivalence of quotient modules obtained from the submodules of functions in $\HM$ (respectively, $\tilde{\HM}$) vanishing of order $r+2$ along $\Z$ forces the Hilbert modules to be unitarily equivalent. It is also  pointed out, for $m\leq r$ and $d\neq m$, that there exist Hilbert modules $\HM$ and $\tilde{\HM}$ which are not equivalent although these quotient modules are. Thus it generalizes one of the main results, namely Theorem 1.6 in \cite{CGOT}, in multi-variable domains as well as provides the main result (Theorem 3.12) in \cite{LOCHM} as a special case. 
\subsection{Notations and Conventions} Before proceeding further, let us fix some notations which will be useful throughout the paper.

\begin{enumerate}
\item $\mathcal{O}(\O)$ denotes the set of all holomorphic functions on $\O$.

\item In this article, we only consider those Hilbert modules $\HM\subset\mathcal{O}(\O)^{\oplus r}$ for which the evaluation mappings $ev_z:\HM\ra\C^r$ are bounded from $\HM$ onto $\C^r$ for $z\in\O$ and the \textit{Gleason property} holds. We also assume that the ring of polynomials $\C[z_1,\hdots,z_m]$ is contained in $\HM$.

\item\label{Res} Let $\HM$ be a Hilbert module over $\A(\O)$ consisting of holomorphic functions on $\O$ and $\HM_0 \subset \HM $ be a subspace which is also a Hilbert module over $\O$. Assume that $\A(\O)$ acts on $\HM$ by point-wise multiplication and $\HM_q$ be the quotient module $\HM\ominus\HM_0$. Let $U \subset \O$ be an open connected subset. Then from the identity theorem for holomorphic functions of several complex variables we have $\HM \simeq_{\A(\O)} {\HM}|_{U}$, $\HM_0 \simeq_{\A(\O)}{\HM_0}|_{U}$, and hence $\HM_q \simeq_{\A(\O)}{\HM_q}|_{U}$ where $\HM|_{U}=\{h|_U: h \in \HM \}$. We, therefore, may cut down the domain $\O$ to a suitable open subset $U$, if necessary, and pretend $U$ to be $\O$.
\item From now on, for any multi-indices $\a=(\a_1,\hdots,\a_d)$, we use following notations
\beq\label{not}
 \del^{\a}(\text{respectively},~\dbar^{\a})
~:=~ \frac{\del^{|\a|}}{\del{z_1}^{\a_1}\cdots\del{z_d}^{\a_d}}\left(\!\text{respectively},~\frac{\dbar^{|\a|}}{\dbar{z_1}^{\a_1}\cdots\dbar{z_d}^{\a_d}}\!\right)
\eeq
unless and otherwise stated, where $\del_i=\frac{\del}{\del z_i}$, $i=1,\hdots,d$.
\end{enumerate}

\section{Preliminaries on \h s}\label{prelim}

Let $E$ be a hermitian holomorphic vector bundle of rank $n$ over a complex manifold $M$ of dimension $m$. A connection on the bundle $E$ is a differential operator $D:\E^0(M,E)\ra \E^{1}(M,E)$ of order $1$ with the defining property
$$D(f\si)=df \otimes\si + f\cdot D\si $$ for any smooth function $f$ on $M$ and $\si \in \E^0(M,E)$ where $df$ stands for usual exterior derivative of $f$, $\E^0(M,E)$ is the set of all smooth sections of $E$ and $\E^1(M,E)$ is the set of all smooth $E$ valued $1$ forms on $M$. It is well known (cf. \cite{PAG}) that there is  unique connection $D$ - the Chern connection - on $E$ which is compatible with both the complex structure and the hermitian metric. It with respect to a local holomorphic frame $s=\{e_1,\hdots,e_n\}$ of $E$, takes the form
\beq D(s)=\del H(s)\cdot H(s)^{-1}\eeq where $H(s)$ is the Gramian matrix of the frame $s$. From now on, by a connection on a {\h} we will mean the Chern connection. The curvature tensor $\K$ of the vector bundle $E\ra M$ \w the Chern connection is an element in $\E^2(M)\otimes \text{Hom}(E,E)$ and takes the form 
\beq\label{locK}\K(s)=\dbar(\del H(s)\cdot H(s)^{-1})\eeq \w the local holomorphic frame $s$ of $E$, where $\E^2(M)$ is the set of all smooth $2$ forms on $M$ and Hom$(E,E)$ is the vector bundle over $M$ with Hom$(E_p,E_p)$ as the fibre over any point $p\in M$. It follows that in a local coordinate system of $M$ one can write $\K$ as \beq\label{locKc}\K(s)=\sum_{i,j=1}^m \K_{i\ov{j}}(s)dz_i\wedge d\ov{z}_j=\sum_{i,j=1}^m \dbar_j(\del_i H(s)\cdot H(s)^{-1})dz_i\wedge d\ov{z}_j,\eeq where $s$ is a local holomorphic frame of $E$.

Note that unlike the curvature tensor $\K$, the connection operator $D$ is not $C^{\infty}$ linear, that is, $D$ is not a bundle map. Nevertheless, it is easy to verify that the commutator of connection with a bundle map is a bundle map. More generally the following fact holds:
\begin{lem}\cite[Lemma 2.10]{CGOT}
Let $E$ and $\tilde{E}$ be $C^{\infty}$ vector bundles over some smooth manifold $X$ with connections $D$ and $\tilde{D}$, respectively. If $\Phi: E \ra \tilde{E}$ is a $C^{\infty}$ bundle map then so is $\tilde{D}\Phi-\Phi D$ as map from $E$ to $\tilde{E}\otimes T^*(X)$, where $T^*(X)$ is the cotangent bundle of $X$.
\end{lem}

\noindent Thus, for bundle maps $\Phi: E \ra \tilde{E}$, there exist bundle maps $\Phi_{z_i},\Phi_{\ov{z}_j}: E \ra \tilde{E}$ satisfying \beq\tilde{D}\Phi-\Phi D=\sum_{i,j=1}^m(\Phi_{z_i}\otimes dz_i+\Phi_{\ov{z}_j}\otimes d\ov{z}_j).\eeq In particular, for a bundle map $\Phi: E \ra E$ we write $$[D,\Phi]:=D\Phi-\Phi D.$$ Let $\Phi(s),\Phi_{z_i}(s),\Phi_{\ov{z}_j}(s)$ be the matrix representation of $\Phi, \Phi_{z_i},\Phi_{\ov{z}_j}$, respectively, \w the local holomorphic frame $\f$. It turns out that
\begin{align}\label{covd1}\Phi_{z_i}(s) &=\del_{z_i}\Phi(s)-[\del_{z_i}H(s)\cdot H(s)^{-1},\Phi(s)]\,\,\,\text{and}\\
&\label{covd2}\Phi_{\ov{z}_i}(s)=\del_{\ov{z}_i}\Phi(s).\end{align}
In the following lemma, we calculate the covariant derivatives of curvature tensor. The proof of the following lemma, for $d=1$, is well known  (cf. \cite[Proposition 2.18]{CGOT}, \cite[Lemma 20]{LOCHM}). Although the similar set of arguments used there with more than one variable yields the proof in our case, we present a sketch of the proof for the sake of completeness.

\begin{lem}\label{cume}
Let $E$ be a hermitian holomorphic vector bundle over $\O$ in $\C^m$ with a fixed holomorphic frame $\f:=\{s_1,\hdots,s_r\}$ whose Gramian matrix is $H$. Then
\begin{enumerate}
\item[(i)] for $1 \leq d \leq m$, $\a,\b \in (\N \cup \{(0)\})^d$, and $i,j=1,\hdots,d$, the $r\times r$ matrices \linebreak $(\K_{i\ov{j}}(S))_{{\bz}^{\a}{\ov{\bz}}^{\b}}$ can be expressed in terms of $H^{-1}$ and ${\del}^{p}{\dbar}^{q}H$ where $z^{\a}={z_1}^{\a_1}\cdots{z_d}^{\a_d}$, ${\ov{z}}^{\b}={\ov{z}_1}^{\b_1}\cdots{\ov{z}_d}^{\b_d}$, $\del^p=\del_1^{p_1}\cdots\del_d^{p_d}$, $\dbar^q=\dbar_1^{q_1}\cdots\dbar_d^{q_d}$ and $0 \leq \sum_{l=1}^d p_l \leq |\a|+1$, $0 \leq \sum_{l=1}^d q_l \leq |\b|+1$, $l=1,\hdots,d$;
\item[(ii)] given $1 \leq d \leq m$, $\a,\b \in (\N \cup \{(0)\})^d$, ${\del}^{\a}{\dbar}^{\b}H$  can be written in terms of $H^{-1}$, ${\del}^{p}H$, ${\dbar}^{q}H$ and $(\K_{i\ov{j}})_{{z}^{r}{\ov{z}}^{s}}$, for $0 \leq p_l \leq \a_l$, $0 \leq q_l \leq \b_l$, $0 \leq \sum_{l=1}^d r_l \leq |\a|-1$, $0 \leq \sum_{l=1}^d s_l \leq |\b|-1$, $l=1,\hdots,d$, $i,j=1,\hdots,d$.
\end{enumerate}
\end{lem}

\begin{proof}
Let $E$, $S$ and $H$ be as above. Then, for $j=1,\hdots,m$, we have
\beq\label{dm} \del_{z_j}H^{-1}=-H^{-1}\cdot\del_{z_j}H\cdot H^{-1},\text{ and}\eeq
\beq\label{dbm} \del_{\ov{z}_j}H^{-1}=-H^{-1}\cdot\del_{\ov{z}_j}H\cdot H^{-1}.\eeq
Now from the definition of curvature we obtain, for $i,j=1,\hdots,d$,
\beqa
\K_{i\ov{j}} &=& \del_{\ov{z}_j}(\del_{z_i}H\cdot H^{-1})\\
&=&\del_{\ov{z}_j}\del_{z_i}H\cdot H^{-1}-\del_{z_i}H\cdot H^{-1}\cdot\del_{\ov{z}_j}H\cdot H^{-1}
\eeqa
which also implies that
\beq\label{ddbm}\del_{\ov{z}_j}\del_{z_i}H=\K_{i\ov{j}}H+
\del_{z_i}H\cdot H^{-1}\cdot \del_{\ov{z}_j}H.\eeq
Then an induction argument using Leibniz rule together with the equations $\eqref{dm}$ and $\eqref{dbm}$ yield the desired expression in (i). Further, (ii) can be obtained as before by using Leibnitz rule and formulas $\eqref{dm},\eqref{dbm} \text{ and }\eqref{ddbm}$ with the help of mathematical induction on $|\a|$ and $|\b|$.
\end{proof}

\section{The Submodule ${\HM}_0$}\label{SM}

Let $\HM\subset\mathcal{O}(\O)^{\oplus r}$ be the Hilbert module over $\A(\O)$ for which the evaluation mappings $ev_z:\HM\ra\C^r$ are bounded from $\HM$ onto $\C^r$ for $z\in\O$ and the \textit{Gleason property} holds. Denote the elements of $\HM$ as $h=(h_1,\hdots,h_r)^{\text{tr}}$ where $h_j\in \A(\O)$, $1 \leq j \leq r$. In this section, we define the submodule $\HM_0$ of $\HM$. So we begin by recalling some elementary definitions regarding complex analytic varieties.

\begin{defn}
A subset $\Z \subset\O$ is called an analytic set if, for any point $p \in \O$, there is a connected open {\nbhd} $U$ of $p$ in $\O$ and finitely many holomorphic functions $\phi_1,\hdots,\phi_d$ on $U$ such that $$U\cap\Z=\{q \in U: \phi_j(q)=0,\text{ }1\leq j\leq d\}.$$
\end{defn}

\begin{defn}\label{sav}
An analytic set $\Z\subset \O$ is said to be regular of codimension $d$ at $p \in \Z$ if there is an open {\nbhd} $U_p \subset \O$ and holomorphic functions $\phi_1,\hdots,\phi_d$ on $U_p$ such that
\begin{enumerate}
\item [(a)] $\Z \cap U_p = \{q \in \O : \phi_1(q)=\cdots=\phi_d(q)=0\}$,
\item [(b)] the rank of the Jacobian matrix of the mapping $q \mapsto (\phi_1(q),\hdots,\phi_d(q))$ at $p$ is $d$.
\end{enumerate}
\end{defn}

An analytic set is said to be irreducible if it can not be decomposed as a union of two analytic sets. It is known in the literature that any smooth analytic set is irreducible if and only if it is connected \w the subspace topology \cite[page 20]{PAG}.


In the following proposition, we point out that such an analytic set $\Z$ is a regular complex submanifold of codimension $d$ in $\O$.

\begin{prop}\cite[page 161]{FHFTCM}\label{defun}
An analytic set $\Z$ is regular of codimension $d$ at $p\in M$ in a complex manifold $M$ of dimension $m$ if and only if there is a complex coordinate chart $(U,\phi)$ of $M$ such that $B:=\phi(U)$ is an open subset of $\C^m$ with $\phi(p)=0$ and $\phi(U\cap\Z)=\{\l=(\l_1,...,\l_m) \in B:\l_1=\cdots=\l_d=0\}$.
\end{prop}

\begin{rem}\label{defc}
In this article, we are interested in smooth irreducible analytic sets $\Z$ of codimension $d$ in some bounded domain $\O$ in $\C^m$. So from the Definition $\ref{sav}$ and the Proposition $\ref{defun}$ we have, for each point $p \in \Z$, there is a coordinate chart $(U,\phi)$ at $p$ of $\O$ satisfying following properties:
\begin{enumerate}
\item [(a)] $\phi(p)=0$ with $\phi(U\cap\Z)=\{\l=(\l_1,...,\l_m) \in B:\l_1=\cdots=\l_d=0\}$,
\item [(b)] the rank of the Jacobian matrix of the mapping $q \mapsto (\phi_1(q),\hdots,\phi_d(q))$ at $p$ is $d$.
\end{enumerate}
\end{rem}

We are now about to define the order of vanishing of a holomorphic function along a smooth analytic set. Our definition is essentially a direct generalization of the definition given in \cite{OQMAM} to define the order of vanishing of a holomorphic function along a smooth complex hypersurface.


\begin{defn}\label{ordofvan}
Let $\O$ and $\Z$ be as above and $f:\O\ra\C$ be a holomorphic function. Then $f$ is said to have zero of order at least $k$ at some point $p\in \Z$ if there exists a coordinate chart $(U,\phi)$ at $p$ of $\O$ satisfying the properties (a) and (b) in the Remark $\ref{defc}$ such that \beq\label{ordv}[f]\in I_{\Z}^{k}\eeq where $[f]$ is the germ of $f$ at $p$ and $I_{\Z}$ is the ideal in $\mathcal{O}_{m,p}$ generated by $[\phi_1],\hdots,[\phi_d]$.
\end{defn}

\begin{rem}\label{coind}
Note that the above definition is independent of the choice of coordinate chart at $p$. Indeed, for two such charts $(U_1,\phi_1)$ and $(U_2,\phi_2)$ with the properties listed in Remark $\ref{defc}$, $\phi_1$ and $\phi_2\circ\phi_1^{-1}$, respectively, induce isomorphisms $\Phi_1:\mathcal{O}_{m,p}\ra\mathcal{O}_{m,0}$ defined by $\Phi_1([g])=[g\circ\phi_1^{-1}]$ and $\Phi_2:\mathcal{O}_{m,0}\ra\mathcal{O}_{m,0}$ with $\Phi([g\circ\phi_1^{-1}])=[g\circ\phi_2^{-1}]$. As a consequence, it turns out that $f$ satisfies $\eqref{ordv}$ if and only if $[f\circ\phi_1^{-1}]\in I_1^{k}$  which is again equivalent to the fact that $[f\circ\phi_2^{-1}]\in I_2^{k}$ where $I_j$ is the ideal generated by the germs $[\l^j_1],\hdots,[\l^j_d]$, $j=1,2$, with local coordinates $\l^j_1,\hdots,\l^j_m$ of $\C^m$ corresponding to $\phi_j$.
\end{rem}

\begin{defn}
The submodule ${\HM}_0$ is defined as
$${\HM}_0:=\{h \in \HM: h_j\,\text{ has zero of order at least }k\text{ at every }q \in \Z,\,\,1\leq j\leq r\}.$$
\end{defn}

\begin{lem}\label{lem}
Let $\O$ be a bounded domain in $\C^m$, $\Z$ be a complex submanifold in $\O$ and $f: \O \ra \C$ be a holomorphic function. Then, for each point $p \in \Z$, $f$ vanishes to order $k$ at $p$ along $\Z$ if and only if 
$$ \del^{\a}_{\l}(f \circ {\phi}^{-1})|_{\phi(U \cap \Z)} := \frac{\del^{|\a|}}{\del{\l_1}^{\a_1}\cdots\del{\l_d}^{\a_d}}(f \circ {\phi}^{-1})|_{\phi(U \cap \Z)}=0\,\,\,\text{for}\,\,\,0 \leq |\a| \leq k-1,$$
where $\a=(\a_1,\hdots,\a_d)$ and $|\a|=\a_1+\cdots+\a_d$, for some coordinate chart $(U,\phi)$ as in the Remark $\ref{defc}$.
\end{lem}

In general, there are no global defining functions $\phi_1,\hdots,\phi_d$ for a smooth irreducible analytic set $\Z$. Since the modules and the submodules of interest can be localized (see at the end of the Section $\ref{intro}$), it is enough to work with an open set $U \subset \O$ intersecting $\Z$. So from now on, we consider a fixed {\nbhd} $U \subset \O$ of $p$ with $U \cap \Z \neq \varnothing$ and defining functions $\phi_1,\hdots,\phi_d$ satisfying conditions (a) and (b) in Remark $\ref{defc}$. Since the Jacobian matrix of the mapping $z \mapsto (\phi_1(z),\hdots,\phi_d(z))$ has rank $d$ at $p$, by rearranging the coordinates in $\C^m$, we can assume that $\mathcal{D}_1(p):=(({\del}_j{\phi}_i|_{p}))^d_{i,j=1}$ is invertible. Note that $\mathcal{D}_1(z)$ is invertible on some {\nbhd} of $p$ in $U$. Abusing the notation, let us denote this {\nbhd} by the same letter $U$. Consider the mapping $\phi:U \ra \phi(U)$ defined as $\phi(z)=(\phi_1(z),\hdots,\phi_d(z),z_{d+1},\hdots,z_m)$ and observe that $\phi$ is a bi-holomorphism from $U$ onto $\phi(U)$ with $\phi(p)=0$ and $\phi(U\cap\Z)=\{\l=(\l_1,...,\l_m) \in \phi(U):\l_1=\cdots=\l_d=0\}$. Thus once we fix a chart as above and pretend $U$ to be $\O$, the submodule $\HM_0$ may be described as $$\HM_0 = \left\{h \in \HM: \frac{\del^{|\a|}}{\del{\l_1}^{\a_1}\cdots\del{\l_d}^{\a_d}}(h_j \circ {\phi}^{-1})(\l)|_{\phi(\Z)}=0\,\,\,\text{for}\,\,\,0 \leq |\a| \leq k-1,\,\,1\leq j \leq r\right\}.$$


At this stage, we introduce a definition that separates the coordinate chart described above among others and will be useful throughout this article.

\begin{defn}\label{ad}
Let $\O$ be a domain in $\C^m$ and $\Z \subset \O$ be a complex submanifold of codimension $d$. Then, for any point $p \in \Z$, we call a coordinate chart $(U,\phi)$ of $\Z$ around $p$ an $\mathit{admissible}$ coordinate chart if the bi-holomorphism $\phi:U \ra \phi(U)$ takes the form $\phi(z)=(\phi_1(z),\hdots,\phi_d(z),z_{d+1},\hdots,z_m)$ with $\phi(p)=0$ and $\phi(U\cap\Z)=\{\l=(\l_1,...,\l_m) \in \phi(U):\l_1=\cdots=\l_d=0\}$ for some holomorphic functions $\phi_1,\hdots,\phi_d$ on $U$.
\end{defn}

Note that even in this local description of the submodule, there is a choice of complementary directions to the submanifold $\Z$ involved. The following proposition ensures that in this local picture two different sets of complementary directions to $\Z$ give rise to equivalent submodules. We must recall some elementary definitions and properties of the ring of polynomial functions on a finite dimensional complex vector space which will be useful in the course of the proof of the following proposition.

For any complex vector space $V$ of dimension $d$, the ring of polynomial functions $\C[V]$ on $V$ consists of all functions $f:V\ra \C$ such that, for any basis $\{e_1,\hdots,e_d\}$ of $V$, $f$ takes the form $f(\a_1 e_1+\cdots+\a_d e_d)=\phi(\a_1,\hdots,\a_d)$ for all $(\a_1,\hdots,\a_d)\in \C^d$ and some polynomial $\phi \in \C[x_1,\hdots,x_d]$. In other words, $f$ is a polynomial into the elements $x_1=e^*_1,\hdots,x_d=e^*_d$ of the dual basis. It is then clear that
\beq\label{eqpf} \C[V] \simeq S(V^*) \simeq \C[x_1,\hdots,x_d]\eeq where $S(V^*)$ is the graded vector space of all symmetric tensors on $V^*$. Note that $\C[V]$ is an algebra over $\C$.


A polynomial function $f$ on $V$ is said to be homogeneous of degree $t$ if $f(\a v)={\a}^t f(v)$ for all $\a \in \C$ and $v \in V$. We denote $\C[V]_t$ the subspace of $\C[V]$ of homogeneous polynomial functions of degree $t$. In particular, $C[V]_0=\C$, $\C[V]_1=V^*$ and $\C[V]_t$ is canonically identified in the first isomorphism in $\eqref{eqpf}$ with the $t$-th symmetric power $S^t(V^*)$, and it can also be identified with the subspace of $\C[x_1,\hdots,x_d]$ generated by the monomials $x_1^{t_1}\cdots x_d^{t_d}$ with $t_1+\cdots+t_d=t$ via the second isomorphism of $\eqref{eqpf}$.

\begin{prop}\label{coch}
Let $\O$ be a bounded domain in $\C^m$, $\Z$ be a complex submanifold in $\O$ and $f: \O \ra \C$ be a holomorphic function. Then, for each point $p \in \Z$ and $n \in \N$, there exists an admissible coordinate chart $(U,\phi)$ of $\O$ at $p$ such that
$$\del_\l^{\a}(f\circ\phi^{-1})(\l)|_{\phi(p)}=0,~0 \leq |\a| \leq n~\text{if and only if}~{\del}^{\a}f(z)|_{p}=0,~0 \leq |\a| \leq n$$
where $\a=(\a_1,\hdots,\a_d) \in (\N \cup \{0\})^d$, $\l=(\l_1,\hdots,\l_m)$ denotes the standard coordinates on $\phi(U)\subset \C^m$.
\end{prop}

\begin{proof}
Let us consider an admissible coordinate system $(U,\phi)$ (Definition $\ref{ad}$) at $p \in \Z \subset \O$ of $\O$, that is, $\phi:U \ra \phi(U)$ defined as $\phi(z)=(\phi_1(z),\hdots,\phi_d(z),z_{d+1},\hdots,z_m)$.


For $q \in U$, let $V_q$ and $V_{\phi(q)}$ be tangent spaces at $q$ and $\phi(q)$ to $(\C^d\times\{0\})\cap U$ and $(\C^d\times\{0\})\cap \phi(U)$, respectively. Denote the standard ordered basis of $V_q$ by $\B_1(q):=\{\frac{\del}{\del z_j}|_q \}_{j=1}^d$ and that of $V_{\phi(q)}$ by $\B_1(\phi(q)):=\{\frac{\del}{\del \l_j}|_{\phi(q)} \}_{j=1}^d$. Consider the linear map $L_1(q):V^*_q \ra V^*_{\phi(q)}$ whose matrix \w the dual bases $\{dz_j\}_{j=1}^d$ and $\{d\l_j\}_{j=1}^d$ of $\B_1(q)$ and $\B_1(\phi(q))$, respectively, is the matrix $((\del_j\phi_i(q)))_{i,j=1}^d$.


In view of the first isomorphism in $\eqref{eqpf}$, we observe that $L_1(q)$ canonically induces linear mappings $L_t(q):S^t(V^*_q)\ra S^t(V^*_{\phi(q)}) $ defined by $$L_t(q)(dz_1^{\a_1}\otimes\cdots\otimes dz_d^{\a_d})=L_1(q)(dz_1)^{\a_1}\otimes\cdots\otimes L_1(q)(dz_d)^{\a_d}$$ where $\a=(\a_1,\hdots,\a_d)\in (\N\cup\{0\})^d$ with $|\a|=\a_1+\cdots+\a_d=t$ and by $dz_j^{\a_j}$ (respectively, by $L_1(q)(dz_j)^{\a_j}$) we mean that the $\a_j$-th symmetric power of $dz_j$ (respectively, $L_1(q)(dz_j)$).


Let $\B_t(q):=\{dz_1^{\a_1}\otimes\cdots\otimes dz_d^{\a_d}:|\a|=t\}$ and $\B_t(\phi(q)):=\{d\l_1^{\a_1}\otimes\cdots\otimes d\l_d^{\a_d}:|\a|=t\}$ be bases for vector spaces $S^t(V^*_q)$ and $S^t(V^*_{\phi(q)})$, respectively, and make them ordered bases \w the order induced by the colexicographic order on the set $\{\a \in (\N \cup \{0\})^d:|\a|=t\}$. Denote the matrix of $L_t(q)$ represented \w the basis $\B_t(q)$ and $\B_t(\phi(q))$ as $\mathcal{D}_t(q)$, for $t\in \N \cup \{0\}$. Note that since $L_t(p)$ is a vector space isomorphism for each $t\in \N \cup \{0\}$ the matrices $\mathcal{D}_t(p)$'s are invertible.


In this set up, we claim, for $z \in U$ with $\phi(z)=\l\in\phi(U)$, that
\beq\label{cvf}
A_{n,\phi}(z)\cdot
\left(
\begin{array}{c}
 f \circ \phi^{-1}(\l)\\
\del_{\l_1}f \circ \phi^{-1}(\l)\\
\vdots\\
\del^{\a}_{\l}f \circ \phi^{-1}(\l)\\
\vdots\\
\del^n_{\l_d}f \circ \phi^{-1}(\l)
\end{array}
\right)
&=&
\left(
\begin{array}{c}
 f(z)\\
\del_1 f(z)\\
\vdots\\
\del^{\a}f(z)\\
\vdots\\
\del^n_d f(z)
\end{array}
\right)
\eeq
where $\del^{\a}_{\l}$ stands for the differential operator $\frac{\del^{|\a|}}{\del\l_1^{\a_1}\cdots\del\l_d^{\a_d}}$, $A_{n,\phi}(z)$ is the block lower triangular matrix with $1$, $\mathcal{D}_1(z)$, $\hdots$, $\mathcal{D}_n(z)$ as the diagonal blocks. Here the order used in writing the coloumn vectors is obtained from the colexicographic order on the set $\{\a \in (\N \cup \{0\})^d: |\a|\leq k\}$.


We prove this claim with the help of mathematical induction on $n$. Note that the base case is a direct consequence of the change of variables formula. Let the equation $\eqref{cvf}$ hold true for $t=l$ with $1 \leq l \leq n$, that is, for $\a=(\a_1,\hdots,\a_d)$ with $|\a|=l$,
\Bea
\del^{\a}f(z) &=& \sum_{|\b|=l}(\mathcal{D}_l(z))_{\a\b}\del_{\l}^{\b}f\circ\phi^{-1}(\l)+
\text{ other terms}
\Eea
where $\mathcal{D}_l(z)$ is the matrix $(((\mathcal{D}_l(z))_{\a\b}))_{|\a|=l,|\b|=l}$. Differentiating both sides of this equation \w the $z_j$-th coordinate and using the Leibnitz rule we have, for an arbitrary but fixed point $q \in U$,
\Bea
\del_j\del^{\a}f(z)|_{z=q} &=& \sum_{|\b|=l}\del_j(\mathcal{D}_l(z))_{\a\b}|_{z=q}\del_{\l}^{\b}f\circ\phi^{-1}(\phi(q))\\ &+&  \sum_{|\b|=l}(\mathcal{D}_l(q))_{\a\b}\left(\sum_{s=1}^d \del_j\phi_s(q)\del_{\l_s}\right)\del_{\l}^{\b}
f\circ\phi^{-1}(\l)|_{\l=\phi(q)} + \del_j(\text{ other terms})\\
&=& \sum_{|\b|=l}(\mathcal{D}_l(q))_{\a\b}\del_{\l}^{\b}\left(\sum_{s=1}^d \del_j\phi_s(q)\del_{\l_s}f\circ\phi^{-1}(\l)\right)\big|_{\l=\phi(q)}\\ &+& (\text{ other terms involving }\del_{\l}^{\b}f\circ\phi^{-1}(\phi(q))\text{ with }|\a|\leq l)
\Eea
Note that the rings of polynomial functions $S(V^*_q)$ and $S(V^*_{\phi(q)})$ can be canonically identified with the algebras of linear partial differential operators with constant coefficients, namely, $\G_q:S(V^*_q)\simeq_{\C}\{\sum_{\a}a_{\a}\del_1^{\a_1}\cdots\del_d^{\a_d}:a_{\a}\in \C\}$ under the correspondence $dz_1^{\a_1}\otimes\cdots\otimes dz_d^{\a_d}\overset{\G_q}{\mapsto} \del^{\a_1}_1\cdots\del^{\a_d}_d$ and similarly, $\G_{\phi(q)}:S(V^*_{\phi(q)})\simeq_{\C}\{\sum_{\a}a_{\a}\del_{\l_1}^{\a_1}\cdots\del_{\l_d}^{\a_d}:a_{\a}\in \C\}$ via the mapping $d\l_1^{\a_1}\otimes\cdots\otimes d\l_d^{\a_d}\overset{\G_{\phi(q)}}{\longmapsto} \del^{\a_1}_{\l_1}\cdots\del^{\a_d}_{\l_d}$. Consequently, it yields that
\Bea
\del^{\a+\varepsilon_j}f(q) &=& \sum_{|\b|=l}(\mathcal{D}_l(q))_{\a\b}\del_{\l}^{\b}\left(\sum_{s=1}^d \del_j\phi_s(q)\del_{\l_s}f\circ\phi^{-1}(\l)\right)\big|_{\l=\phi(q)}\\ &+& (\text{ other terms involving }\del_{\l}^{\b}f\circ\phi^{-1}(\phi(q))\text{ with }|\a|\leq l)\\
&=&(\G_{\phi(q)}L_1(q)\G_q^{-1}(\del_1))^{\a_1}\cdots
(\G_{\phi(q)}L_1(q)\G_q^{-1}(\del_j))^{\a_j}\cdots
(\G_{\phi(q)}L_1(q)\G_q^{-1}(\del_d))^{\a_d}\\ && (\G_{\phi(q)}L_1(q)\G_q^{-1}(\del_j))f\circ\phi^{-1}(\l)\big|_{\l=\phi(q)}\\ &+& (\text{ other terms involving }\del_{\l}^{\b}f\circ\phi^{-1}(\phi(q))\text{ with }|\a|\leq l)\\
&=& (\G_{\phi(q)}L_1(q)\G_q^{-1}(\del_1))^{\a_1}\cdots
(\G_{\phi(q)}L_1(q)\G_q^{-1}(\del_j))^{\a_j+1}\cdots
(\G_{\phi(q)}L_1(q)\G_q^{-1}(\del_d))^{\a_d}\\ && f\circ\phi^{-1}(\l)\big|_{\l=\phi(q)}+(\text{ other terms involving }\del_{\l}^{\b}f\circ\phi^{-1}(\phi(q))\text{ with }|\a|\leq l)\\
&=& \sum_{|\b|=l+1}(\mathcal{D}_{l+1}(q))_{(\a+\varepsilon_j)\b}\del_{\l}^{\b}f\circ\phi^{-1}(\phi(q))\\ &+& (\text{ other terms involving }\del_{\l}^{\b}f\circ\phi^{-1}(\phi(q))\text{ with }|\a|\leq l)
\Eea
which verifies the claim. Thus, $A_{n,\phi}(z)$ is invertible if and only if $\mathcal{D}_1(z)$, $\hdots$, $\mathcal{D}_n(z)$ are simultaneously invertible which is the case for $z \in U$. Hence it completes the proof.
\end{proof}

\begin{rem}
We point out that the proposition above can also be obtained from the multi-variable version of the Faa di Bruno formula for composite functions as suggested by the anonymous referee. However, the advantage of the proof above is in computing the invertible matrix $A_{n,\phi}(z)$ satisfying the equation \eqref{cvf}. We will use this invertible matrix again in Section \ref{JB} which is why we prefer to keep this proof here than the one by Faa di Bruno formula. Nevertheless, we provide a brief sketch of this proof following Faa di Bruno formula for sake of completeness.

We now elaborate upon this. Let $\O$, $\Z$, and $f$ be as in the Proposition \ref{coch}, and $(U,\phi)$ be an admissible coordinate chart at $p$. The multi-variable Faa di Bruno formula, for the composite function $f\circ\phi^{-1}$ and multi-indices $\mu,\mu^{ij},\nu,\nu^{ij}$ of length $m$, states that
$$\del_{\l}^{\nu}(f\circ\phi^{-1})(\l)=\sum_{1\leq |\mu|\leq n}(\del^{\mu}f)(\phi^{-1}(\l))\sum_{\mu^{ij}\neq 0}\sum_{\nu^{ij}}C(\mu^{ij},\nu^{ij})\prod_{i=1}^j(\del_{\l}^{\nu^{ij}}\phi^{-1}(\l))^{\mu^{ij}}$$
where $1\leq i\leq j\leq n=|\nu|$, $(\del_{\l}^{\nu^{ij}}\phi^{-1}(\l))^{\mu^{ij}}$ denotes a monomial product, and for each $j$, the summation is subject to the restriction $\sum_{i=1}^j\l^{ij}=\l,~\sum_{i=1}^j|\mu^{ij}|\nu^{ij}=\nu.$ $C(\mu^{ij},\nu^{ij})$ are the combinatorial coefficients which are explicitly known in \cite{FADIBRUNO}. Denote $\a''=(\a_{d+1},\hdots,\a_m)$. Note that being an admissible coordinate chart, $\phi^{-1}_{d+1},\hdots,\phi^{-1}_m$ are independent on first $d$ variables and hence $\del_{\l}^{\a}\phi^{-1}_k(\l)=0$ whenever $\a''=0$ and $d+1\leq k\leq m$.

Let us consider $\nu$ such that $\nu''=0$ which is the case in the previous proposition and observe from the condition $\sum_{i=1}^j|\mu^{ij}|\nu^{ij}=\nu$ that all $\nu^{ij}$ also satisfy $(\nu^{ij})''=0$ as $|\mu^{ij}|>0$. It turns out, for $i\leq j$ and $\nu$ with $\nu''=0$, that the only $\mu^{ij}$, which can occur in the formula above, are those with $(\mu^{ij})''=0$. Also, it implies that all such $\mu$, whose $\mu''=0$, can occur in the right hand side of the formula above completing the proof of the forward direction of the proposition above. Finally, a similar argument as above with $f\circ\phi^{-1}$ and $\phi$ in place of $f$ and $\phi^{-1}$, respectively, completes the proof of Proposition \ref{coch}. 
\end{rem}

From the proposition above and Remark $\ref{coind}$ we have another characterization of the submodule ${\HM}_0$ as follows:
$${\HM}_0=\{h \in \HM:{\del_1}^{\a_1}\cdots{\del_d}^{\a_d}(h_j)|_{\Z}=0,\,0 \leq |\a| \leq k-1,\,\,1 \leq j \leq r\}.$$

\begin{rem}\label{remres}
Note that following $\eqref{Res}$ in Section $\ref{intro}$, it is enough to restrict the module $\HM$ and the submodule $\HM_0$ to an admissible coordinate chart $(U,\phi)$ around some point $p \in \Z \subset \O$. Also, it can be seen that the unitary equivalence classes of these submodules remain the same under the change of variables $\phi$. We now elaborate upon this fact.
\end{rem}

Let us consider the module, $\phi^*(\HM|_U)$ which is, by definition, $$\phi^*(\HM|_U):=\{f|_{U}\circ\phi^{-1}:f \in \HM\}$$ and note that it is a module over $\A(\O)$ with the module action $g\cdot(f|_{U}\circ\phi^{-1}):=(gf)|_{U}\circ\phi^{-1}$, for $g \in \A(\O)$. It is evident that the modules $\phi^*(\HM|_U)$ and $\HM$ are isomorphic via the isomorphism $\Phi:\HM \ra \phi^*(\HM|_U)$ defined by $f \mapsto f|_{U}\circ\phi^{-1}$. So, defining an inner product as $$\<f|_{U}\circ\phi^{-1},g|_{U} \circ\phi^{-1}\>_{\phi^*(\HM|_U)}:=\<f,g\>_{\HM},$$ it can be seen that $\phi^*(\HM|_U)$ is unitarily equivalent to $\HM$ as Hilbert modules. Since $\HM$ is a reproducing kernel Hilbert module with a reproducing kernel, say, $K$ so is $\phi^*(\HM|_U)$ with the kernel function $K'$ defined by $K'(u,v)=K(\phi^{-1}(u),\phi^{-1}(v))$, for $u,v \in \phi(U)$. Also, the multiplication operators $M_{z_1},\hdots,M_{z_m}$ on $\HM$ are simultaneously unitarily equivalent to $M_{u_1},\hdots,M_{u_m}$ on $\phi^*(\HM|_U)$. 
Finally, $\eqref{Res}$ in Section $\ref{intro}$ together with the Proposition \ref{coch} ensure that the submodules $\HM_0$ and $\phi^*R(\HM_0)$ are also unitarily equivalent via the same map as mentioned earlier. Consequently, we have the following Proposition.

\begin{prop}\label{eococh}
Let $\O$ be a bounded domain in $\C^m$, $\Z$ be a complex connected submanifold in $\O$ and $\HM^1$, $\HM^2$ be two Hilbert modules of rank $r$ over $\A(\O)$. Let $\HM^1_0$ and $\HM^2_0$ be submodules of $\HM^1$ and $\HM^2$, respectively, consisting of holomorphic functions vanishing of order at least $k$ along $\Z$. Assume that $(U,\phi)$ is an admissible coordinate system around some point $p \in \Z$. Then $\HM^1_0$ is unitarily equivalent to $\HM^2_0$ as Hilbert modules if and only if $\phi^*(\HM^1_0|_U)$ is unitarily equivalent to $\phi^*(\HM^2_0|_U)$. In other words, the following diagram commutes.

$$\begin{CD}
\HM^1_0 @>R>> \HM^1_0|_U @>\Phi>> \phi^*(\HM^1_0|_U)\\
@VVV @VVV @VVV\\
\HM^2_0 @>R>> \HM^2_0|_U @>\Phi>> \phi^*(\HM^2_0|_U)
\end{CD}$$
\end{prop}

\begin{rem}\label{submodasmp}
In view of the previous proposition and Definition \ref{ad} we note that it is enough to consider the submanifold $\Z$ as the coordinate plane of codimension $d$. Therefore, from now on we shall only consider domains $\O$ which contain the origin, $$\Z=\{z\in\O:z_1=\hdots=z_d=0\}$$ and the submodule $\HM_0$ as
$${\HM}_0=\{h \in \HM:{\del_1}^{\a_1}\cdots{\del_d}^{\a_d}(h_j)|_{\Z}=0,\,0 \leq |\a| \leq k-1,\,\,1 \leq j \leq r\}.$$
\end{rem}


\section{Quotient Module ${\HM}_q$}\label{QM}

Let $\O$ and $\Z$ be as described in the previous section. Consider the exact sequence \beq\label{exact} 0 \ra \HM_0 \overset{i}{\ra} \HM \overset{P}{\ra} \HM_q \ra 0\eeq where $i$ is the inclusion map and $P$ is the quotient map. Then ${\HM}_q:=\HM/{\HM}_0=\HM \ominus {\HM}_0$. For $f \in \A(\O)$ and $h \in \HM$, we define the module action on $\HM_q$ as
\beq fP(h)=P(fh)\eeq making it a Hilbert module. Here we mean $(fh_1,\hdots,fh_r)$ by $fh$. In this section, we are interested in obtaining a model for these quotient modules $\HM_q$. 

Following the ideas of \cite{OQMAM}, we first describe the jet construction relative to the submanifold $\Z$. Let $A=\{\a \in (\N \cup  \{0\})^d:|\a|< k\}$ be an ordered set equipped with the colexicographing ordering, $\{\e_{\a}\}_{\a \in A}$ be the standard ordered basis of $\C^{|A|}$, $\{\sigma_i\}_{i=1}^r$ be the standard ordered basis of $\C^r$ and recall that $\del_1,...,\del_d$ are the partial derivative operators \w $z_1,...,z_d$ variables, respectively. For $h \in \HM$, we note that $h=(h_1,\hdots,h_r)^{\text{tr}}= \sum_{i=1}^r h_i \otimes \si_i$, and define
$$ \textbf{h}:=\sum_{i=1}^r\left(\sum_{\a\in A}\del^{\a} h_i \otimes \e_{\a}\right)\otimes \sigma_i.$$ Consider the space $J(\HM):=\{\textbf{h}:h \in \HM\} \subset \HM\otimes (\C^{|A|}\otimes\C^r)$.  Here and throughout this article, we follow the notation below for the tensor product: $$A\otimes B=((Ab_{ij}))_{i,j=1}^q$$ for any matrix $A$ and $B=((b_{ij}))_{i,j=1}^q$. Consequently, we have the mapping \beq\label{jcon} J:\HM \ra J(\HM) \text{ defined by } h \mapsto \textbf{h}.\eeq Since $J$ is injective we define an inner product on $J(\HM)$ making $J$ to be an unitary transformation as follows $$\<J(h_1),J(h_2)\>_{J(\HM)}:=\<h_1,h_2\>_{\HM}.$$

Since each evaluation mapping $ev_z:\HM\ra\C^r$ is bounded, $\HM$ is a \rkhs implying that $J(\HM)$ is also a \rkhs with the reproducing kernel $JK$ as computed in the following proposition.

\begin{prop}\label{jet ker}
$J(\HM)$ is a reproducing kernel Hilbert space with the reproducing kernel $JK:\O\times\O\ra M_{|A|r}(\C)$ defined as
\beq\label{matJK}
 (JK)^{\a\b}_{ij}(z,w)=\del^{\a}\dbar^{\b} K_{ij}(z,w)\,\,\,\text{for}\,\,\,\a,\b \in A,\,\,1\leq i,j \leq r,
\eeq
where $JK(z,w)$ is an $r\times r$ block matrix with $(((JK)^{\a\b}_{ij}(z,w)))_{\a,\b\in A}$ as the $|A|\times |A|$ block and $(JK)^{\a\b}_{ij}(z,w)$ is the $\a\b-$th element of the $ij-$th block.
\end{prop}

\begin{proof}
We begin with the observation that 
$$ JK(.,w)\e_{\b}\otimes\si_j = \sum_{\a\in A}\sum_{i=1}^r\del^{\a}\dbar^{\b}K_{ij}(.,w)\e_{\a}\otimes\si_i = J(\dbar^{\b}K(.,w)\si_j)$$
which shows that $JK(.,w)\e_{\b}\otimes\xi \in J(\HM)$ for all $\xi\in\C^r$ and $w\in\O$. It remains to show that $JK$ has the reproducing property which is, by definition,
\beq\<\textbf{h},JK(.,w)\z\>_{J(\HM)}=\<\textbf{h}(w),\z\>_{\C^{|A|}\otimes\C^r},\eeq for $\textbf{h} \in J(\HM)$ and $\z \in M_{|A|r}(\C)$. Note that, for $w \in \O$, $\a \in A$ and $1 \leq j \leq r$, 
\Bea
\<\textbf{h}, JK(.,w)\e_{\b}\otimes\si_j\>_{J(\HM)} &=& \<\textbf{h},J(\dbar^{\b}K(.,w)\si_j)\>_{J(\HM)}\\
&=& \< h,\dbar^{\b}K(.,w)\si_j\>_{\HM}\\
&=& \del^{\b}h_j
\Eea
which completes the proof as $\del^{\a} h_j(w)=\<\textbf{h}(w),\e_{\a}\otimes\sigma_j\>_{\C^{|A|}\otimes\C^r}$.
\end{proof}

We now define an action of $\A(\O)$ on $J(\HM)$ so that $J(\HM)$ becomes a module over $\A(\O)$ and $J$ turns out to be a module isomorphism. For $f \in \A(\O)$ and $\textbf{h} \in J(\HM)$, the module action $J_f: J(\HM) \ra J(\HM)$ is defined by $J_f(\textbf{h}):=\mathcal{J}(f)\otimes I_r\textbf{h}$ where $\mathcal{J}(f) \in M_{|A|}(\C)$ is complex matrix defined as follows
\beq\label{modac}
\mathcal{J}(f)_{\a\b}:= {\a \choose \b}\del^{\a-\b}f := {\a_1 \choose \b_1}\cdots{\a_d \choose \b_d}\del^{\a-\b}f
\eeq 
with $\a=(\a_1,\hdots,\a_d)$ and $\b=(\b_1,\hdots,\b_d)$ and $I_r$ is the identity matrix of size $r$. Note that this is a lower triangular matrix and it takes the following matrix form
\[ \J(f)=
    \left(
    \begin{array}{ccccc}
        f                                    \\
      & \ddots           &   & \text{\huge0}\\
    \vdots  &    \mathcal{J}(f)_{\a\b}           & \ddots               \\
      &  &   &            \\
     \del_d^{k-1} f &  \hdots    &  \hdots &   & f
    \end{array}
    \right)
\]
Using the Leibniz rule, for $1 \leq i \leq r$, we have
\Bea
J(f\cdot h_i) &=& \sum_{\a \in A}\del^{\a}(f \cdot h_i)\otimes\e_{\a}\\
&=& \sum_{\a \in A}\sum_{\b_1=0}^{\a_1}\cdots\sum_{\b_d=0}^{\a_d}\left({\a_1 \choose \b_1}\cdots{\a_d \choose \b_d}\del_1^{\a_1-\b_1}\cdots\del_d^{\a_d-\b_d}f\cdot \del_1^{\b_1}\cdots\del_d^{\b_d}h_i\right)\otimes \e_{\a}\\
&=& \J(f)\cdot \textbf{h}_i\Eea 
which shows that $J(f\cdot h)=\mathcal{J}(f)\otimes I_r\cdot\textbf{h}=J_f(\textbf{h})$ implying that $J$ is a module isomorphism.

As in the case of Hilbert submodule $\HM_0$ of the Hilbert module $\HM$ it is clear that the subspace $$J(\HM)_0:=\{\textbf{h}\in J(\HM):\textbf{h}|_{\Z}=0\}$$ is a submodule of $J(\HM)$. Let $J(\HM)_q$ be the quotient module obtained by taking an orthogonal complement of $J(\HM)_0$ in $J(\HM)$, that is, $J(\HM)_q:=J(\HM)\ominus J(\HM)_0$. The following theorem provides the equivalence of two quotient modules $\HM_q$ and $J(\HM)_q$.

\begin{thm}\label{qejq}
${\HM}_q$ and $J(\HM)_q$ are isomorphic as modules over $\A(\O)$.
\end{thm}

\begin{proof}
Let us begin by pointing out that $D:=\{\dbar^{\a}K(.,w)\sigma_i:w \in \Z,1\leq i \leq r,\a \in A\}$ is a spanning set for $\HM_q$. Indeed, using the reproducing property of $K$, we note for $\z \in \C^r$, $w \in \Z$, $\a \in A$, and $h \in {\HM}_0$, that
$$\<h,\dbar^{\a}K(.,w)\z\>=\<\del^{\a}h(w),\z\>=0$$implying that $\HM_0\subset\text{span}D^{\perp}$. Also, a similar computation yields that span$D\subset \HM_0^{\perp}$.

From the equation \eqref{jcon} it follows that $J(D)= \{JK(.,w)\varepsilon_{\a}\otimes\sigma_i: \a \in A, 1\leq i \leq r, w \in \Z\}.$ As above, for $\textbf{h} \in J(\HM)$, $1 \leq i \leq r$, and $w \in \O$, from the reproducing property of $JK$ it turns out that
\beq
\<\textbf{h},JK(.,w)\varepsilon_{\a}\otimes\sigma_i\>_{J(\HM)} = \<\textbf{h}(w),\varepsilon_{\a}\otimes\sigma_i\>_{\C^{|A|r}} = \del^{\a} h_i(w),~ \a \in A,
\eeq
justifying that $J(D)$ spans $J({\HM})_q$. Thus $J(\HM_q) = J(\text{span}D) = \text{span}J(D) = J(\HM)_q$.

In course of completion of the proof, it remains to check that $J$ is a module isomorphism from ${\HM}_q$ onto $J(\HM)_q$. In other words, we need to verify the following identity
$$J \circ P \circ M_f = (JP)\circ J_f\circ J,$$ for $f \in \A(\O)$, which is equivalent to show that $$J{M^*_f}P = {J^*_f}(JP)J$$ where $JP:J(\HM)\ra J(\HM)_q$ is the orthogonal projection operator. Since it amounts to show that $J$ intertwines the module actions on $D$ and both $P$ and $JP$ are identity on $D$ and $J(D)$, respectively, it is enough to prove that $$J{M^*_f} = {J^*_f}J,\,\,\,\text{for}\,\,f \in \A(\O), \,\,\,\text{on}\,\,D.$$


Let $\a=(\a_1,\cdots,\a_d) \in A $, $1 \leq i \leq r$ and $\dbar^{\a}K(.,w)\sigma_i \in D$. For $f \in \A(\O)$, $w \in \Z\subset\O$, we have \beq {M^*_f}K(.,w)\sigma_i=\ov{f(w)}K(.,w)\sigma_i.\eeq Differentiating both sides of the equation above and using the induction on the degree of differentiation we obtain
\beq {M_f}^*\dbar^{\a}K(.,w)\sigma_i &=& \sum_{\b_1=0}^{\a_1}\cdots\sum_{\b_d=0}^{\a_d}{\dbar_1}^{\a_1-\b_1}\cdots{\dbar_d}^{\a_d-\b_d}\ov{f(w)}{\dbar_1}^{\b_1}\cdots{\dbar_d}^{\b_d}K(.,w)\sigma_i.\eeq Therefore, the equation \eqref{jcon} together with the Proposition \ref{jet ker} yield that
$$J({M_f}^*\dbar^{\a}K(.,w)\sigma_i)= JK(.,w)(\mathcal{J}(f)(w))^*(\varepsilon_{\a}\otimes\sigma_i).$$
Finally, for $\textbf{h} \in J(\HM)$, $\z \in \C^{|A|}\otimes\C^r$ and $w \in \O$, we have
$$
\<\textbf{h},{J^*_f}JK(.,w)\cdot \z\>_{J(\HM)}
= \<\textbf{h}(w),\mathcal{J}(f)(w)^*\z\>_{\C^{|A| \times r}}
= \<\textbf{h},JK(.,w)\mathcal{J}(f)(w)^*\z\>_{J(\HM)}
$$
completing the proof.
\end{proof}

\begin{rem}
Note that as mentioned in \cite{OQMAM} the Theorem $\ref{qejq}$ is equivalent to the fact that the following diagram of exact sequences is commutative.
$$\begin{CD}
0 @>>> \HM_0 @>i>> \HM @>P>> \HM_q @>>> 0\\
  @.     @VVV @VVV @VVV\\
0 @>>> J(\HM)_0 @>i>> J(\HM) @>JP>> J(\HM)_q @>>> 0
\end{CD}$$
\end{rem}


In \cite{TRK}, it was shown that for a \rkhs $\H$ with scalar valued reproducing kernel $K$ on some set $W$, the restriction of $K$ on a subset $W_1$ of $W$ is also a reproducing kernel and restriction of $K$ to $W_1$ constitutes a \rkhs which is isomorphic to the quotient space $\H \ominus \H_0$ where $\H_0:=\{f \in \H:f|_{W_1}=0\}$. Here, adopting the proof from \cite{TRK} for our case with vector valued kernel, we have the following theorem.  Since this result is well known (Theorem 3.3, \cite{OQMAM}) for the case while the codimension of the submanifold, $\Z$, is one and using the techniques used in that proof in a similar way the following theorem can be obtained, we omit the proof.

\begin{thm}\label{qeres}
The normed linear space $J(\HM)|_{\Z}$ is a Hilbert space and the Hilbert spaces $J(\HM)_q$ and $J(\HM)|_{\Z}$ are unitarily equivalent. Consequently, the reproducing kernel $K_1$ for $J(\HM)|_{\Z}$ is the restriction of the kernel $JK$ to the submanifold $\Z$. Moreover, $J(\HM)_q$ and $J(\HM)|_{\Z}$ are isomorphic as modules over $\A(\O)$.
\end{thm}

\begin{thm}\label{qm}
The quotient module ${\HM}_q$ is equivalent to the module $J(\HM)|_{\Z}$ over $\A(\O)$.
\end{thm}

\begin{proof}
It is obvious from Theorem $\ref{qejq}$ and Theorem $\ref{qeres}$.
\end{proof}

\begin{rem}\label{diagonal}
Let us consider a reproducing kernel Hilbert module $\HM$ over $\A(\D^m)$ and $\Delta_d=\{z=(z_1,\hdots,z_m)\in\D^m: z_1=\cdots=z_d\}$ be the connected submanifold. Suppose that $\HM_0$ is the submodule defined as follows
$$\HM_0=\{f\in\HM:\del_1^{\a_1}\cdots\del_{d-1}^{\a_{d-1}}f|_{\Delta_{d}}=0:\a=(\a_1,\hdots,\a_{d-1})\in(\N\cup\{0\})^{d-1},0\leq|\a|\leq k-1\}.$$
It follows from Proposition \ref{coch} that the submodule $\HM_0$ is independent of the choice of complementary directions to $\Delta_d$ (for example, $\HM_0$ is isomorphic with the submodule of functions $f$ in $\HM$ such that $\del_1^{\a_1}\cdots\del_{d-1}^{\a_{d-1}}f|_{\Delta_{d}}=0$ for $\a\in(\N\cup\{0\})^{d-1}$ with $0\leq |\a|\leq k-1$). Consider the quotient module $\HM_q=\HM\ominus\HM_0$ and note that as in the proof of Theorem \ref{qm} it can be shown that $\HM_q$ is isomorphic to the module of jets $J_1(\HM)|_{\Delta_d}$ restricted to $\Delta_d$ where $$J_1(\HM)=\left\{\sum_{0\leq|\a|\leq k-1}\del_1^{\a_1}\cdots\del_{d-1}^{\a_{d-1}}f\otimes\varepsilon_{\a}:\a=(\a_1,\hdots,\a_{d-1})\in(\N\cup\{0\})^{d-1}\right\}.$$ 

On the other hand, from Proposition \ref{eococh} and Theorem \ref{qm} we have that the quotient module $\HM_q$ is isomorphic to $J(\phi^*(\HM))|_{\Delta_d}$ as modules where $\phi:\D^m\ra\C^m$ is the bi-holomorphism $\phi(z)=(z_1-z_d,\hdots,z_{d-1}-z_d,z_d,\hdots,z_m)$ onto it's image, $\phi^*(\HM)=\{f\circ\phi^{-1}:f\in\HM\}$ and $J(f\circ\phi^{-1})=\sum_{0\leq|\a|\leq k-1}\del_{\l_1}^{\a_1}\cdots\del_{\l_{d-1}}^{\a_{d-1}}(f\circ\phi^{-1})\otimes\varepsilon_{\a}$ with $\l_j=\phi_j$, $j=1,\hdots,d-1$ as $(\D^m,\phi)$ is an admissible coordinate chart of $\D^m$ (Definition \ref{ad}). Therefore, $J_1(\HM)|_{\Delta_d}$ is also isomorphic to $J(\phi^*(\HM))|_{\Delta_d}$ as modules. Note that this new model $J_1(\HM)|_{\Delta_d}$ for the quotient modules $\HM_q$ of above kind is more canonical than the one obtained from the jet construction relative to a coordinate plane.

We also point out that the construction depicted above can be performed to any linear varieties which possess a global admissible coordinate chart.
\end{rem}

\textbf{Example.} For $\l > 0$, let $\H^{(\l)}(\D)$ be the Hilbert space of holomorphic functions on $\D$ with the reproducing kernel $K^{(\l)}(z,w)=(1-z \ov{w})^{-\l}$ for $z,w \in \D$ with $\{e^{(\l)}_n(z):=c_n^{-\frac{1}{2}}z^n:n \geq 0 \}$ as a complete orthonormal set in $\H^{(\l)}(\D)$ where $c_n$ are the $n$-th coefficient of the power series expansion of $(1-|z|^2)^{- \l}$, $$ c_n = {-{\l} \choose n} = \frac{\l(\l+1)(\l+2)\cdots (\l+n-1)}{n!}=\frac{(\l)_n}{n!}.$$ Let us recall that for $\l > 0$, the natural action of polynomial ring $\C[z]$ on each Hilbert space $\H^{(\l)}(\D)$ makes it into a Hilbert module over $\C[z]$. We also point out that, for $\l > 1$, $\H^{(\l)}(\D)$ becomes a Hilbert module over the disc algebra $\A(\D)$.

It is well known that the product of two reproducing kernels is also a reproducing kernel \cite[Page no. 8]{TRK}. For $\l=({\l}_1,\hdots,{\l}_m)$ with ${\l}_i> 0$, $i=1,\hdots,m$, let us consider the Hilbert space $\H^{(\l)}(\D^m):= \H^{({\l}_1)}(\D)\otimes \cdots \otimes \H^{({\l}_m)}(\D)$ with the natural choice of complete orthonormal set $\{e^{({\l}_1)}_{i_1}(z)\otimes \cdots \otimes e^{({\l}_m)}_{i_m}(z):i_j \geq 0, j=1,\hdots,m\}$. $\H^{(\l)}(\D^m)$ naturally possesses, under the identification of the functions $z_1^{i_1}\cdots z_m^{i_m}$ on $\D^m:=\D\times\cdots\times\D$, an obvious reproducing kernel $$K^{(\l)}(z,w):=\prod_{j=1}^m(1-z_j\ov{w}_j)^{-{\l}_j}$$ on $\D^m$. Furthermore, the natural action of $\C[\boldsymbol{z}]$ on $\H^{(\l)}(\D^m)$ makes it a Hilbert module over $\C[\boldsymbol{z}]$, for ${\l}_j>0$, $j=1,\hdots,m$, where by $\C[\boldsymbol{z}]$ we mean $\C[z_1,\hdots,z_m]$.


Let us now consider the subspace $\H^{(\l)}_0$ consisting of holomorphic functions in $\H^{(\l)}$ which vanish to order $2$ along the diagonal $\Delta:=\{(z_1,\hdots,z_m)\in \D^m:z_1=\cdots=z_m\}$, that is, following the definition given in Section $\ref{SM}$, $$\H^{(\l)}_0=\{f \in \H^{(\l)}(\D^m):f=\del_1 f= \cdots =\del_m f = 0\,\,\text{on}\,\,\Delta\}.$$

\noindent Following the remark above the quotient module $\H^{(\l)}_q:=\H^{(\l)}(\D^m)\ominus\H^{(\l)}_0$ can be identified with the reproducing kernel Hilbert module over $\A(\D)$ with the reproducing kernel $$K_q(\bz,\bw)|_{\Delta}=JK^{(\l)}(\bz,\bw)|_{\Delta}=((\del_i\dbar_jK^{(\l)}(\bz,\bw)|_{\Delta}))_{i,j=0}^{m}.$$

\begin{rem}\label{rem}
Let us now clarify the module action of $\A(\O)$ on the quotient module ${\HM}_q$ before proceeding further. To facilitate this action we, following \cite[page 384]{OQMAM}, consider the algebra of holomorphic functions on $\O$ taking values in $\C^{|A|}$ with $A=\{\a \in (\N \cup \{0\})^d:|\a|<k\}$,
$$J\A(\O):=\{\J f:f \in \A(\O)\} \subset \A(\O) \otimes M_{|A|}(\C)$$ with the usual matrix multiplication, namely, $(\J f\cdot\J g)(z):=\J f(z)\J g(z)$. It is clear from $\eqref{modac}$ that $J(\HM)|_{\Z}$ is a module over the algebra $J\A(\O)|_{\Z}$ obtained by restricting $J\A(\O)$ to $\Z$. Note that $J$ defines an algebra isomorphism from $\A(\O)$ onto $J\A(\O)$ and intertwines the restriction operators $R_1:\A(\O)\ra\A(\O)|_{\Z}$ and $R_2:J\A(\O)\ra J\A(\O)|_{\Z}$. Consequently, $J:\A(\O)|_{\Z}\ra J\A(\O)|_{\Z}$ is also an algebra isomorphism. So $J(\HM)|_{\Z}$ can be thought of as a Hilbert module over $J\A(\O)|_{\Z}$.


On the other hand, observe that the inclusion $i:\Z\ra \O$ induces a map $i^*:J\A(\O)\ra J\A(\O)|_{\Z}$ defined by $i^*(\J f)(z)=\J f(i(z))$, for $z \in \Z$. Then $J(\HM)|_{\Z}$ can be made a module over $J\A(\O)$ by defining the module map as follows:
$$\J f\otimes I_r\cdot\textbf{h}|_{\Z}:=i^*(\J f\otimes I_r)\textbf{h}|_{\Z}.$$ Thus, recalling the fact that $J$ defines an algebra isomorphism between $\A(\O)$ and $J\A(\O)$, we can think of $J(\HM)_q$ as a module over $\A(\O)$. Moreover, from the equation above it can be seen, for $f=z_i$ with $i=d+1,\hdots,m$, that the module action $J_f=M_{\J f\otimes I_r}$ on $J(\HM)|_{\Z}$ becomes the multiplication by $z_i$. Thus, $J(\HM)|_{\Z}$ can be thought of as a \rkhs with the reproducing kernel $JK|_{\Z}$ on which the multiplication operators corresponding to the coordinate functions of $\Z$ are obtained from the module actions $J_{z_i}$, $i=d+1,\hdots,m$. 
\end{rem}

Since the similar construction can be done for the Hilbert modules $\HM \in \mathrm B_r(\O)$ with submodules $\HM_0$ consisting of holomorphic functions in $\HM$ vanishing along $\Z$ to order $k$, it is natural to ask whether the quotient spaces arising from such submodules are in $\mathrm B_{|A|r}(\Z)$ where $|A|$ is the cardinality of the set $A$ of all multi-indices $\a=(\a_1,\hdots,\a_d)$ with $|\a|<k$. In the following theorem, we give an affirmative answer to this question for a large collection of Hilbert modules.

\begin{thm}\label{qmodinCDcls}
Let $\O \subset \C^m$ be a bounded domain containing the origin and $\Z\subset \O$ be the coordinate plane defined by $\Z =\{z=(z_1,\hdots,z_m) \in \O:z_1=\cdots=z_d=0\}$. We also assume that $\HM$ is a \rkhs with the reproducing kernel $K$ such that $\C[z_1,\hdots,z_m]\subset\HM$. Then the quotient Hilbert space $\HM_q$ lies in $\mathrm B_{|A|r}(\Z)$ provided $\HM \in \mathrm B_r(\O)$ where $|A|$ is the cardinality of $A=\{\a \in (\N \cup \{0\})^d:|\a|<k\}$.
\end{thm}

\begin{proof}
Let us denote $z'=(z_1,\hdots,z_d)$ and $z''=(z_{d+1},\hdots,z_m)$. In view of Theorem $\ref{qm}$ and Remark $\ref{rem}$, it is enough to prove that the module of jets, $J(\HM)$ restricted to $\Z$ is in $B_{|A|r}(\Z)$. Since $\C[z_1,\hdots,z_m]\subset\HM$ and $J(\HM)$ has a reproducing kernel, the evaluation functionals are bounded and onto. So from Theorem \ref{gleason} it is enough to show that, for any $\textbf{h}\in J(\HM)|_{\Z}$ and $(0,w'')\in\Z$, $\textbf{h}(0,w'')=0$ if and only if $\textbf{h}=(z_{d+1}-w_{d+1})\textbf{h}_1+\cdots+(z_m-w_m)\textbf{h}_{m-d}$ on $\Z$ for some $\textbf{h}_1,\hdots,\textbf{h}_{m-d}\in J(\HM)|_{\Z}$ which follows from the lemma below.
\end{proof}

\begin{lem} Let $\textbf{h}\in J(\HM)|_{\Z}$ with $\textbf{h}=J(h)$ and $(0,w'')\in\Z$ be an arbitrary point. Assume also that $\textbf{h}(0,w'')=0$. Then there exist $h_1,\hdots,h_{m-d}\in\HM$ such that
\beq\label{3.3.10eqn1}
h=\sum_{j=1}^{m-d}(z_{d+j}-w_{d+j})h_j(z',z'')+\sum_{|\a|=k}z^{\a}g_{\a}
\eeq
for some $g_{\a}\in \HM$, $\a=(\a_1,\hdots,\a_d)$ with $|\a|=k$.
\end{lem}

\begin{proof}
We prove the desired identity in \eqref{3.3.10eqn1} with the help of mathematical induction on $|\a|$. Since $\textbf{h}(0,w'')=0$ it follows from the Gleason property of $\HM$ (as $\HM\in B_r(\O)$) that $$h(z',z'')=\sum_{i=1}^{d}z_ig_i(z',z'')+\sum_{j=1}^{m-d}(z_{d+j}-w_{d+j})g_{d+j}(z',z''))$$ for some $g_1,\hdots,g_m\in\HM$. So it takes care of the base case with $h_j=g_{d+j}$, $j=1,\hdots,m-d$.

Assume that the claim is true for $|\a|\leq k-2$. Let $\textbf{h}\in J(\HM)$ with $\textbf{h}=J(h)$ and $\textbf{h}(0,w'')=0$. From the definition of $J$ in \eqref{jcon}, it follows that $\del^{\a}h(0,w'')=0$ for all $\a\in A$. Consequently, from the induction hypothesis it turns out that 
\beq\label{3.3.10eqn2} 
h(z',z'') = \sum_{j=1}^{m-d}(z_{d+j}-w_{d+j})g_j(z',z'')+\sum_{|\a|=k-1}z^{\a}g_{\a}
\eeq
for some $g_{\a}\in \HM$, $\a=(a_1,\hdots,a_d)$ with $|\a|=k-1$. For any $\a\in A$ with $|\a|=k-1$, applying the differential operator $\del^{\a}$ to the both sides of the equation \eqref{3.3.10eqn2} we have that $$\del^{\a}h(0,z'') = \sum_{j=1}^{m-d}(z_{d+j}-w_{d+j})\del^{\a}g_j(0,z'')+cg_{\a}(0,z'')$$ for some non-zero constant $c$. But since $\del^{\a}h(0,w'')=0$ so is $g_{\a}(0,w'')=0$. Therefore, from the Gleason property of $\HM$ it can be seen that $$g_{\a}(z',z'')=\sum_{i=1}^dz_ig_{\a}^i(z',z'')+\sum_{j=1}^{m-d}(z_{d+j}-w_{d+j})\tilde{g}_{j\a}(z',z'')$$ for some $g_{\a}^i,\tilde{g}_{j\a}\in\HM$ for $i=1,\hdots,d$ and $j=1,\hdots,m-d$. Now substituting $g_{\a}$'s in \eqref{3.3.10eqn2} we have that $$h(z',z'')=\sum_{j=1}^{m-d}(z_{d+j}-w_{d+j})\left(g_j+\sum_{|\a|= k-1}(z')^{\a}\tilde{g}_{j\a}\right)(z',z'')+\sum_{|\a|=k-1,i=1}^dz^{\a+\varepsilon_i}g^i_{\a}$$ which completes the proof.
\end{proof}

\begin{rem}
Note that the assumption $\C[z_1,\hdots,z_m]\subset\HM$ implies that the reproducing kernel $JK$ for the Hilbert module $J(\HM)$ is non-degenerate (cf. \cite[Definition 4.2]{GBKCD}).
\end{rem}

We note that the above theorem provides examples of quotient modules which are in the Cowen-Douglas class. 
We now provide another important application of the lemma above which will be useful in the next section.

\begin{prop}\label{generalized eigenspace}
Let $\O \subset \C^m$ be a bounded domain containing the origin and $\Z\subset \O$ be the coordinate plane defined by $\Z =\{z=(z_1,\hdots,z_m) \in \O:z_1=\cdots=z_d=0\}$. We also assume that $\HM$ is a \rkhs with the reproducing kernel $K$ such that $\C[z_1,\hdots,z_m]\subset\HM$ and $\HM\in\mathrm B_r(\O)$. Then $\{\dbar^{\a}K(.,(0,w''))\varepsilon_j:0\leq |\a|\leq k-1,~(0,w'')\in \Z,~1\leq j\leq r\}$ forms a basis for $\cap_{i=d+1}^m\ker(M^*_{z_i}-\ov{w}_i)\cap\ker(M^*_{z''})^{\b}$ where $\ker(M^*_{z''})^{\b}$ is $\cap_{|\b|=k}\ker((M^*_{z_1})^{\b_1}\cdots(M^*_{z_d})^{\b_d})$ and $\b=(\b_1,\hdots,\b_d)$.
\end{prop}

\begin{proof}
We first observe from the reproducing property that, for any $w\in\O$ and $1\leq i\leq r$, $$M^*_{z_i}(K(.,w)\varepsilon_j)=\ov{w}_iK(.,w)\varepsilon_j.$$ So it follows, for any $w\in\O$ and $\a_i,\b_i\in\N\cup\{0\}$, that $$(M^*_{z_i}-\ov{w}_i)^{\b_i}\dbar_i^{\a_i}K(.,w)\varepsilon_j=\b_i!{\a_i\choose\b_i}\dbar_i^{\a_i-\b_i}K(.,w)\varepsilon_j.$$ Consequently, we have, for $w\in\O$, $1\leq j\leq r$ and $\a=(\a_1,\hdots,\a_d)\in(\N\cup\{0\})^d$, that $$(M^*_{z_1}-\ov{w}_1)^{\b_1}\cdots(M^*_{z_d}-\ov{w}_d)^{\b_d}(\dbar^{\a}K(.,w)\varepsilon_j)=\b!{\a\choose\b}\dbar^{\a-\b}K(.,w)\varepsilon_j$$ implying that the set $\{\dbar^{\a}K(.,(0,w''))\varepsilon_j:0\leq |\a|\leq k-1,~(0,w'')\in \Z,~1\leq j\leq r\}$ is contained in $\cap_{i=d+1}^m\ker(M^*_{z_i}-\ov{w}_i)\cap\ker(M^*_{z''})^{\b}$.

Let $f$ be any element perpendicular to the set $\{\dbar^{\a}K(.,(0,w''))\varepsilon_j:0\leq |\a|\leq k-1,~(0,w'')\in \Z,~1\leq j\leq r\}$. It follows from the reproducing property that $\del^{\a}f(0,w'')=0$ for all $\a$ with $0\leq|\a|\leq k-1$. From the lemma above we have that $f$ lies in $\oplus_{i=d+1}^m(M_{z_i}-w_i)\HM\oplus_{|\b|=k}(M_{z_1})^{\b_1}\cdots(M_{z_d})^{\b_d}\HM$ which is perpendicular to $\cap_{i=d+1}^m\ker(M^*_{z_i}-\ov{w}_i)\cap\ker(M^*_{z''})^{\b}$. Therefore, $f=0$.
\end{proof}

\section{Jet Bundle}\label{JB}

This section is devoted to provide geometric invariants of quotient modules $\HM_q$ introduced in the previous section. To begin with, since the Hilbert module $\HM$ in $\mathrm B_r(\O)$, $\HM$ gives rise to a {\h} $E$ with the frame $\{K(.,\ov{w})\sigma_1,\hdots,K(.,\ov{w})\sigma_r:w \in \O^* \}$ on $\O^*$. Now to make calculations simpler let us consider the map $\mathit{c}:\O \ra \O^*$ defined by $w\mapsto \ov{w}$ and pull back the bundle $E$ to a vector bundle over $\O$. We denote this new bundle with the same letter $E$ and note that $E$ is a {\h} over $\O$ with the global holomorphic frame $\f:=\{s_1(w),\hdots,s_r(w):w \in \O \}$ with $s_j(w):=K(.,\ov{w})\sigma_j$, $1 \leq j \leq r$. Correspondingly, we have $\del^{\a}s_j(w)=\del^{\a}K(.,\ov{w})\sigma_j$, $1 \leq j \leq r$, $\a \in A$ with $A=\{\a \in (\N \cup \{0\})^d:|\a|<k\}$.


Following the ideas in \cite{OQMAM} we attempt, in this section, to describe the jet bundle of the vector bundle $E\ra\O$ relative to a connected complex submanifold $\Z$ of codimension $d\geq 1$. Our interest is to investigate unitary invariants of the quotient module $\HM_q$ obtained from the submodules $\HM_0$ (cf. Section \ref{SM}) assuming that $\HM \in \mathrm B_{r}(\O)$. Therefore, following the Remark \ref{submodasmp}, it is enough to consider the submanifolds 
$$\Z := \{z=(z_1,\cdots,z_m) \in \O : z_1=\cdots=z_d=0\}.$$
We define the jet bundle $J^{(k)}E$ of order $k$ of $E$ relative to $\Z$ on $\O$ by declaring $\{\del^{\a} \f\}_{\a \in A}$ as a frame for $J^{(k)}E$ on $\O$ where we mean by $\del^{\a}\f$ the ordered set of sections $\{\del^{\a}s_1,\hdots,\del^{\a}s_r\}$, $\a \in A$. Since we have a global frame on $J^{(k)}E$ we do not need to worry about the transition rule.

At this point, we should note that this construction of jet bundle depends on the choice of the complementary direction to $\Z$ which is, a priori, not unique. For instance, two different choices of admissible coordinate charts (Definition \ref{ad}) $(U_1,\phi_1)$ and $(U_2,\phi_2)$ would give rise to two jet bundles $J^{(k)}_iE$, $i=1,2$ with global frames $\{\del^{\a}_{\l_i}s_j:\a\in A,1\leq j\leq r\}$ for $i=1,2$, respectively. A priori it is not clear if these two bundles are equivalent. Nevertheless, it can be seen from the following proposition.

\begin{prop}
Let $(U_1,\phi_1)$ and $(U_2,\phi_2)$ be two admissible coordinate charts of $\O$ around some point $p \in \Z$. Then two jet bundles $J^{(k)}_1E$ and $J^{(k)}_2E$ obtained as above \w $(U_1,\phi_1)$ and $(U_2,\phi_2)$, respectively, are equivalent holomorphic vector bundles over $U_1\cap U_2$.

\end{prop}

\begin{proof}
In fact, from Proposition $\ref{coch}$ it is clear, for a holomorphic frame $\f=\{s_1,\hdots,s_r\}$ of $E$ on $U_1\cap U_2$, that on a small enough {\nbhd} $U$ of $p$ in $U_1\cap U_2$ we have, for $i=1,2$,
$$A_{k-1,\phi_i}(z)\cdot
\begin{pmatrix}
s^i_1(\l) & \cdots & s^i_r(\l)\\
\del_{\l_1}s^i_1(\l) & \cdots & \del_{\l_1}s^i_r(\l)\\
\vdots &  & \vdots\\
\del^{\a}_{\l}s^i_1(\l) & \cdots & \del^{\a}_{\l}s^i_r(\l)\\
\vdots &  & \vdots\\
\del_{\l_d}^{k-1}s^i_1(\l) & \cdots & \del_{\l_d}^{k-1}s^i_r(\l)
\end{pmatrix}=
\begin{pmatrix}
s_1(z) & \cdots & s_r(z)\\
\del_1 s_1(z) & \cdots & \del_1 s_r(z)\\
\vdots &  & \vdots\\
\del^{\a} s_1(z) & \cdots & \del^{\a} s_r(z)\\
\vdots &  & \vdots\\
\del_d^{k-1}s_1(z) & \cdots & \del_d^{k-1}s_r(z)
\end{pmatrix},$$ for $z\in U$ and $\l_i \in \phi_i(U)$, where $\l_i=(\l_{i1},\hdots,\l_{id})$, $\a=(\a_1,\cdots,\a_d)$, and $s^i_j=s_j\circ\phi_i^{-1}$, $1 \leq j \leq r$. Since $A_{k-1,\phi_i}(z)$, for $i=1,2$ and $z\in U$, are invertible (Proposition $\ref{coch}$) we can see that $(A_{k-1,\phi_1}(z)\circ A_{k-1,\phi_2}(z)^{-1})\otimes I_r$ is the desired bundle map where $I_r$ is the identity matrix of order $r$.
\end{proof}


Now in course of completing our construction to make the jet bundle $J^{(k)}E$ a hermitian holomorphic vector bundle, we need to put a hermitian metric on $J^{(k)}E$ extending the metric on $E$. To this extent, if $H(w)=((\langle s_i(w),s_j(w)\rangle_{E}))_{i,j=1}^r$ is the metric on $E$ over $\O$ then the hermitian metric on $J^{(k)}E$ \w the frame $\{\del^{\a} \f\}_{\a \in A}$ is given by the Grammian $JH:=((JH_{\a\b}))_{\a,\b\in A}$ with $r\times r$ blocks $$JH_{\a\b}(w):=((\langle\del^{\a} s_i(w), \del^{\b} s_j(w)\rangle))_{i,j=1}^r\,\,\,\text{for}\,\,\a,\b \in A, w \in \O.$$
This completes our construction of the jet bundle.

\begin{rem}\label{rjbi}
(i) Assume that $\C[z_1,\hdots,z_m]\subset\HM$. Then it follows from Theorem \ref{qmodinCDcls} that the quotient module $\HM_q\in\mathrm B_{r|A|}(\Z)$. Also, we note that the {\h} $\mathscr{E}|_{\Z}\ra\Z$ obtained from $J(\HM)|_{\Z}$ is equivalent to the jet bundle $J^{(k)}E|_{\Z}\ra \Z$ of $E$ relative to $\Z$. To facilitate, $K=((K_{ij}))_{i,j=1}^r$ be the reproducing kernel associated to $\HM$. From the preceding construction it follows that the metric for the jet bundle is given by the formula
$$\langle\del^{\a} K(.,\ov{w})\sigma_i,\del^{\b} K(.,\ov{w})\sigma_j\rangle=\del^{\a}\dbar^{\b} K_{ji}(\ov{w},\ov{w})\,\,\,\text{for}\,\,w\in\Z,~\a,\b \in A,~1 \leq i,j \leq r.$$
On the other hand, the bundle $\mathscr{E}|_{\Z}\ra\Z$ associated to $J(\HM)|_{\Z}$ is spanned by the global holomorphic frame $\{JK(.,\ov{w})\varepsilon_{\a}\otimes\varepsilon_j:\a\in A,~1\leq j\leq r\}$. So we define the bundle map $\Phi:\mathscr{E}|_{\Z}\ra J^{(k)}E|_{\Z}$ by $JK(.,\ov{w})\varepsilon_{\a}\otimes\varepsilon_j\mapsto \del^{\a} K(.,\ov{w})\sigma_j$, for $w\in \Z,\a\in A$ and $1\leq j\leq r.$ Since $J$ is the unitary module map $J:\HM \ra J(\HM)$ it then follows that $\Phi$ is an isometry. Therefore, the vector bundle $\mathscr{E}|_{\Z}$ is unitarily equivalent to $J^{(k)}E|_{\Z}$.

(ii) We point out that, in our notation of jet bundle, $J^{(1)}E$ is nothing else but the bundle $E$ itself, although in literature $J^{(1)}E$ means the first jet bundle of the vector bundle $E$. We maintain these notations following \cite{OQMAM,EQHMII}.

(iii) Note that the action of the algebra $\A(\O)$ on the module $J(\HM)$ defines, for every $f \in \A(\O)$, a holomorphic bundle map $\Psi_f:J^{(k)}E\ra J^{(k)}E$ whose matrix representation \w the frame $J(\f):=\{\sum_{\a \in A}\del^{\a} s^1\otimes \e_{\a},\hdots,\sum_{\a \in A}\del^{\a} s_r\otimes \e_{\a}\}$ is the matrix $\J(f)\otimes I_r$ where $\J(f)$ is as in $\eqref{modac}$ and $I_r$ is the identity matrix of order $r$. Thus, $\Psi_f$ induces an action of $\A(\O)$ on the holomorphic sections of the jet bundle $J^{(k)}E$ defined by \beq\label{modacsec}(f\cdot \sigma)(w) := \Psi_f(\sigma(w)),\eeq for $f \in \A(\O)$, $w \in \O$ and $\sigma$ is a holomorphic section of $J^{(k)}E$. Therefore, we observe that the question of determining the equivalence classes of modules $J(\HM)$ is same as understanding the equivalence classes of the jet bundles $J^{(k)}E$ with an additional assumption that the equivalence bundle map is also a module map on holomorphic sections over $\A(\O)$. Hence it is natural to give the following definition (Definition 4.2, \cite{EQHMII}).
\end{rem}

\begin{defn}\label{jbi}
Two jet bundles are said to be equivalent if there is an isometric holomorphic bundle map which induces a module isomorphism of the class of holomorphic sections.
\end{defn}

\subsection{Main results from Jet bundle}

In order to find geometric invariants of quotient modules we first investigate the simple case, $d=k=2$. We show here that the curvature is the complete set of unitary invariants of the quotient module $\HM_q$ for the Hilbert module $\HM$ when $r=1$. In this case, we give a computational proof to depict the actual picture behind the general result which we will prove later in this subsection. Although the line of idea of the proof for $k=2$ essentially is the same as in \cite{EQHMII}, in our case calculations become more complicated as here we have to deal with more than one transversal directions to $\Z$. Thus, our results extend most of the results of the paper \cite{OQMAM}, \cite{EQHMII} as well as those from a recent paper \cite{ALOTCD}.


Since $d=2$ we have that $\Z = \{(z_1,\hdots,z_m)\in \O:z_1=z_2=0\}$. Consequently, $(0,0,z_3,\hdots,z_m)$ is the coordinates of $\Z$ in $U$. Now let us begin with a line bundle $E$ over $\O^*$ with the real analytic metric $G$ which possesses the following power series expansion
\beq G(z',z'')=\sum_{\a,\b=0}^{\infty}G_{\a\b}(z''){z'}^{\a}\ov{z'}^{\b}\eeq where $(z',z'')\in \O^*$, $\a,\b$ are multi-indices, ${z'}^{\a}={z_1}^{\a_1}{z_2}^{\a_2}$, $\ov{z'}^{\b}={\ov{z}_1}^{\b_1}{\ov{z}_2}^{\b_2}$ and $z''=(z_3,\hdots,z_m)$.


\begin{lem}\label{trial}
Let $\O \subset \C^m$ be a bounded domain and $\Z$ be a complex connected submanifold of $\O$ of codimension $2$. Suppose that $\K$ and $\tilde{\K}$ are the curvature tensors of line bundles $E$ and $\tilde{E}$ \w the hermitian metric $\rho$ and $\tilde{\rho}$ of $E$ and $\tilde{E}$, respectively. Then $\K$ and $\tilde{\K}$ are equal on $\Z$ if and only if there exist holomorphic functions $\s_{00},\s_{10},\s_{01}$ on $\Z$ such that
\beq \label{em}((\tilde{\rho}_{\a\b}))_{|\a|,|\b|=0}^1=\Psi\cdot(({\rho}_{\a\b}))_{|\a|,|\b|=0}^1\cdot{\Psi}^*\eeq on $\Z$ where $\rho_{\a\b}~ (\tilde{\rho}_{\a\b},~ \text{respectively})=\del^{\a}\dbar^{\b}\rho~(\del^{\a}\dbar^{\b}\tilde{\rho},~\text{respectively})$ with $\a,\b\in(\N\cup\{0\})^2$ and $\Psi$ is the $3 \times 3$ matrix
\beq\label{msi}
\Psi = \begin{pmatrix}
\s_{00} & 0 & 0\\
\s_{10} & \s_{00} & 0\\
\s_{01} & 0 & \s_{00}\\
\end{pmatrix}.
\eeq
\end{lem}
Before going into the proof of the lemma let us give an application of it as follows.

\begin{thm}
Suppose that $\HM$ and $\tilde{\HM}$ are in $B_1(\O)$. Then the quotient modules $\HM_q$ and $\tilde{\HM}_q$ are isomorphic if and only if the corresponding curvature tensors $\K$ and $\tilde{\K}$ of the line bundles $E$ and $\tilde{E}$, respectively, are equal on $\Z$.
\end{thm}

\begin{proof}
In fact, Theorem $\ref{qm}$ provides that equivalence of $\HM_q$ and $\tilde{\HM}_q$ is same as the equivalence of $J(\HM)|_{\Z}$ and $J(\tilde{\HM})|_{\Z}$. So let us begin with an isometric module map $\Psi : J(\HM)|_{\Z} \ra J(\tilde{\HM})|_{\Z}$. Since $\Psi$ intertwines the module action $\Psi$ is of the form given in $\eqref{msi}$. Since $\Psi$ intertwines the module action $\Psi$ satisfies the following equation
$$ \Psi \begin{pmatrix}
f & 0 & 0\\
\del_1f & f & 0\\
\del_2 f & 0 & f\\
\end{pmatrix}
=\begin{pmatrix}
f & 0 & 0\\
\del_1f & f & 0\\
\del_2 f & 0 & f\\
\end{pmatrix}\Psi,$$
for every $f\in\A(\O)$. Putting suitable functions in the equation above we see that $\Psi$ takes the form given in $\eqref{msi}$. Moreover, being an isometry, $\Psi$ satisfies \beq\label{ek} JK|_{\Z} = \Psi\cdot J\tilde{K}|_{\Z}\cdot \Psi^*\eeq which is equivalent to saying that $\Psi$ satisfies the identity $\eqref{em}$ on $\Z$ as, for $z \in \Z$, $\rho(z)$ is nothing but $K(z,z)$. Then the Lemma $\ref{trial}$ proves the necessity part.


Conversely, since $\K=\tilde{\K}$ on $\Z$, it follows from Lemma $\ref{trial}$ that $\Psi$ is of the form given in $\eqref{msi}$ and satisfies $\eqref{ek}$. Consequently, $\Psi$ is an isometry from $J(\HM)|_{\Z}$ onto $J(\tilde{\HM})|_{\Z}$ and intertwines the module action.
 \end{proof}

\begin{proof}[Proof of Lemma $\ref{trial}$]
Note that if $\rho$ and $\tilde{\rho}$ are equivalent hermitian metrics satisfying the equation \eqref{em}, then the equality of $\K$ and $\tilde{\K}$ on $\Z$ is a direct consequence of the definition of curvature form of a \h. So we only consider the forward direction in which case we have to find $\s_{00},\s_{10},\s_{01}$ holomorphic on $\Z$ such that $\eqref{em}$ holds assuming that $\K$ and $\tilde{\K}$ are equal along $\Z$.


Let $\r=r\cdot\rho$ and $\G=\log r$. Then $\G$ is real analytic function on $\O$. Therefore, $\Gamma$ can be expanded in the power series
\beq\label{GPE}\G(z',z'')=\sum_{\a,\b=0}^{\infty}\G_{\a\b}(z''){z'}^{\a}\ov{z'}^{\b} \eeq
where $\a,\b$ are multi-indices, ${z'}^{\a}={z_1}^{\a_1}{z_2}^{\a_2}$, $\ov{z'}^{\b}={\ov{z_1}}^{\b_1}{\ov{z_2}}^{\b_2}$ and $z''=(z_3,\hdots,z_m)$. Paraphrasing the assumption on $\K$ and $\tilde{\K}$ we have that $\del_i\dbar_j\G=0$, for $1 \leq i,j \leq m$, along $\Z$. We now separate out this into following three different cases.
\begin{itemize}
\item[\textbf{I.}] ($\del_i\dbar_j\G=0$ along $\Z$, for $i=1,2$, $j=3,\hdots,m$) It follows from $\eqref{GPE}$ that $\del_1\dbar_j\G|_{\Z}=0$ for $j=3,\hdots,m$ is equivalent to the fact that $\G_{(1,0)(0,0)}$ is holomorphic on $\Z$. Similarly considering the case with $i=2$,  we get $\G_{(0,1)(0,0)}$ is also holomorphic on $\Z$.
\item[\textbf{II.}] ($\del_i\dbar_j \G =0$ along $\Z$, for $i,j=1,2$) In view of the equation $\eqref{GPE}$, it yields that $\G_{\a\b}=0$ on $\Z$ for $\a,\b\in\{(1,0),(0,1)\}$.
\item[\textbf{III.}] ($\del_i\dbar_j\G=0$ along $\Z$, for $i,j=3,\hdots,m$) In this last case, we have $\del_i\dbar_j\G|_{\Z}=0$, $i,j=3,\hdots,m$ which together with power series expansion of $\G$ yield that $\del_i\dbar_j\G_{(0,0)(0,0)}=0$, for $i,j=3,\hdots,m$, on $\Z$. Since $\Z$ is a complex submanifold with coordinates $z=(0,0,z_3,\hdots,z_m)\in\Z$ the above equations together imply that $\G_{(0,0)(0,0)}(z'')=\s_1(z'')+\ov{\s_2}(z'')$,  for $z''\in\Z$ and some holomorphic functions $\s_1,\s_2$ on $\Z$.
\end{itemize}

\noindent Now, substituting the above coefficients in the equation $\eqref{GPE}$ and noting that $\G$ is real valued, we have
$$
\G(z',z'')= \s_1+\b_1 z_1+\eta_1 z_2+\ov{\s_2}+\ov{\b_2 z_1}+\ov{\eta_2 z_2}+\text{ (terms of degree}\geq 3)
$$
where $\s_i,\b_i,\eta_i$, $i=1,2$, are holomorphic functions on $\Z$. Since $\G$ is a real valued function $\G=\frac{\G+\ov{\G}}{2}$ and hence we have
\beq
\G(z',z'')= \s+\b z_1+\eta z_2+\ov{\s}+\ov{\b z_1}+\ov{\eta z_2}+\text{ (terms of degree}\geq 3)
\eeq
where $\s=\frac{\s_1+\s_2}{2}$, $\b=\frac{\b_1+\b_2}{2}$ and $\eta=\frac{\eta_1+\eta_2}{2}$. So from the definition of $\G$ we can write
\Bea
r &=& \exp\G\\
&=& |\exp\s|^2\cdot|(1+\b z_1+\ov{\b z_1}+|\b|^2 z_1\ov{z}_1+\cdots)|^2\cdot(|1+\eta z_2+\ov{\eta z_2}+|\eta|^2 z_2\ov{z}_2+\cdots)|^2\cdots\\
&=& |\exp\s|^2\cdot(1+\b z_1+\eta z_2+\ov{\b z_1}+\ov{\eta z_2}+|\b|^2 z_1\ov{z}_1+\b\ov{\eta}z_1\ov{z}_2+\ov{\b}\eta\ov{z}_1 z_2+|\eta|^2 z_2\ov{z}_2+\cdots)
\Eea
Thus, putting the above expression of $r$ in $\r=r\cdot\rho$ and equating the coefficients of $\r$ and $\rho$ we see that $\s_{00},\s_{10}$ and $\s_{01}$ with $$\s_{00}=\exp\s,\s_{10}=\exp\s\b,\s_{01}=\exp\s\eta $$ verify the equation $\eqref{msi}$.
\end{proof}

It would be nice if one could carry forward the arguments used in the proof of Lemma $\ref{trial}$ to achieve similar results in the case of arbitrary order of vanishing of vector valued functions. However, for general $k$, it would be cumbersome to continue the calculation done in the above lemma. On the other hand, application of normalized frames makes the calculations simpler and enables us to get a conceptual proof in the general case as well. We adopt the idea of using a normalized frame from \cite{ALOTCD} in our case to provide the geometric invariants for quotient modules using jet bundle construction relative to a smooth complex submanifold of codimension $d$. To this extent, the following theorem provides the required dictionary between the analytic theory and geometric theory for quotient modules obtained from submodules consisting of vector valued holomorphic functions on $\O$ vanishing along a smooth complex submanifold of codimension $d$. As mentioned earlier in this section, from now on we will assume (without lose of generality) that $\O\subset\C^m$ contains the origin and the $\Z$ is the coordinate plane $$\Z := \{z=(z_1,\cdots,z_m) \in \O : z_1=\cdots=z_d=0\}.$$

\begin{thm}\label{anvsgeo}
Let $\O$ and $\Z$ be as above. Assume that both $\HM$ and $\tilde{\HM}$ contain $\C[z_1,\hdots,z_m]$ and are in $\mathrm B_r(\O)$. Then the quotient modules ${\HM}_q$ and $\tilde{\HM}_q$ are equivalent as modules over $\A(\O)$ if and only if the jet bundles $J^{(k)}E|_{\Z}$ and $J^{(k)}\tilde{E}|_{\Z}$ are equivalent where $E$ and $\tilde{E}$ are the {\h s} over $\O$ corresponding to Hilbert modules $\HM$ and $\tilde{\HM}$, respectively.
\end{thm}

\begin{proof}
Let $\HM_q$ and $\tilde{\HM}_q$ be unitarily equivalent. Then there exists a unitary $U:J(\HM)|_{\Z}\ra J(\tilde{\HM})|_{\Z}$ such that $$U(\J f\otimes I_r)|_{\Z}(\textbf{h}|_{\Z})=(\J f\otimes I_r)|_{\Z}U(\textbf{h}|_{\Z}),~\text{for}~\textbf{h}\in J(\HM).$$ In particular, for $f=z_i$, $i=1,\hdots,m$, $\cap_{i=d+1}^m\ker(M^*_{z_i}-\ov{w}_i)\cap\ker(M^*_{z''})^{\b}$ is preserved by $U$ where $\ker(M^*_{z''})^{\b}$ is $\cap_{|\b|=k}\ker((M^*_{z_1})^{\b_1}\cdots(M^*_{z_d})^{\b_d})$ and $\b=(\b_1,\hdots,\b_d)\in(\N\cup\{0\})^d$. Therefore, it follows from Proposition \ref{generalized eigenspace} that the map $\Phi:J^{(k)}E|_{\Z}\ra J^{(k)}\tilde{E}|_{\Z}$ defined by $\Phi(\del^{\a}s_i(0,w''))=U(\del^{\a}\tilde{s}_i(0,w''))$ is a jet bundle isomorphism.


Conversely, since span$\{\del^{\a}s_i(w):1\leq i\leq r,~w\in\Z,~0\leq |\a|\leq k-1\}$ and span$\{\del^{\a}\tilde{s}_i(w):1\leq i\leq r,~w\in\Z,~0\leq |\a|\leq k-1\}$ are dense in $\HM_q$ and $\tilde{\HM}_q$, respectively, any jet bundle isomorphism between $J^{(k)}E|_{\Z}$ and $J^{(k)}\tilde{E}|_{\Z}$ defines a unitary module map between $\HM_q$ and $\tilde{\HM}_q$.
\end{proof}

We now determine the geometric invariants of quotient modules $\HM_q$ by studying the geometry of the jet bundles $J^{(l)}E|_{\Z}$, for $0 \leq l \leq k$. Before proceeding further, let us recall a fact from complex analysis.

\begin{lem}\label{lem1}
Let $\O \subset \C^m$ be a domain and $f(z,w)$ be a function on $\O \times \O$ which is holomorphic in $z$ and anti-holomorphic in $w$. If $f(z,z)=0$ for all $z \in \O$, then $f(z,w)=0$ identically on $\O$.
\end{lem}

Since this lemma is well-known \cite[Proposition 1]{DATOBS} we omit the proof. We use this lemma several times in the proof of the following theorems.


We observe that a {\h} can not have a holomorphic orthonormal frame in general. Instead one can have (Lemma 2.4 of \cite{CGOT}) a holomorphic frame on a {\nbhd} of a point which is orthonormal at that point. Then using the technique of the proof of Lemma 2.4 in \cite{CGOT} in a similar way, we have the following existence of normalized frame of a {\h} over $\O$ along a submanifold of codimension at least $d$ in $\O$. In the following proposition we use the notation $z=(z',z'')$ where $z'=(z_1,\hdots,z_d)$ and $z''=(z_{d+1},\hdots,z_m)$.

\begin{prop}\label{prop1}
Let $E$ be a {\h} of rank $r$ over a bounded domain $\O\subset \C^m$ containing $0$ and $\Z$ be as above. Then there is a holomorphic frame $\f(z',z'')=\{s_i(z',z'')\}_{i=1}^r$ on a {\nbhd} of the origin in $\O$ such that $((\<\del^{\a} s_i(0,z''),s_j(0,0)\>))_{i,j=1}^r$ is the zero matrix for $\a \in (\N\cup\{0\})^d$ and $((\<s_i(0,z''),s_j(0,0)\>))_{i,j=1}^r$ is the identity matrix on $\Z$.
\end{prop}

We say a frame is \textbf{\textit{normalized at origin along $\Z$}} if it satisfies the properties in the above proposition.

\begin{thm}\label{jean}
Let $\O,\Z$ be as above. Assume that both $\HM$ and $\tilde{\HM}$ contain $\C[z_1,\hdots,z_m]$ and are in $\mathrm B_1(\O)$. Then $\HM_q$ and $\tilde{\HM}_q$ are unitarily equivalent as modules over $\A(\O)$ if and only if $\del^{\a}\dbar^{\b}\norm{\tilde{s}}^2=\del^{\a}\dbar^{\b}\norm{s}^2 $ on $\Z$ for all $\a,\b \in A=\{\a \in (\N\cup\{0\})^d:|\a|< k\}$ where $\{s(z)\}$ and $\{\tilde{s}(z)\}$ are frames of the line bundles $E$ and $\tilde{E}$ on $\O$ associated to the Hilbert modules $\HM$ and $\tilde{\HM}$, respectively, normalized at origin along $\Z$.
\end{thm}

\begin{proof}
Following Theorem $\ref{anvsgeo}$ it is enough to prove that there exists a jet bundle isomorphism $\Phi:J^{(k)}E|_{\Z}\ra J^{(k)}\tilde{E}|_{\Z}$ if and only if $\del^{\a}\dbar^{\b}\norm{\tilde{s}}^2=\del^{\a}\dbar^{\b}\norm{s}^2 $ on $\Z$ for all $\a,\b \in A$.


We start with the necessity. Let $\Phi:J^{(k)} E|_{\Z}\ra J^{(k)}\tilde{E}|_{\Z}$ be a jet bundle isomorphism which can be represented by a complex matrix $((\phi_{\a\b}))_{\a,\b\in A}$ in $M_{|A|}(\C)$ \w the frames $\{\del^{\a} s(0,z'')\}_{\a \in A}$ and $\{\del^{\a}\tilde{s}(0,z'')\}_{\a \in A}$ where $\phi_{\a\b}$ are holomorphic functions on $\Z$. Consequently, by Definition $\ref{jbi}$, we have the following two matrix equations on $\Z$:
\begin{equation}\label{jme}
((\langle\del^{\a} s,\del^{\b} s\rangle))_{\a,\b \in A} = ((\phi_{\g\d}))_{\g,\d \in A}((\langle\del^{\a}\tilde{s},\del^{\b}\tilde{ s}\rangle))_{\a,\b \in A}((\phi_{\g\d}))_{\g,\d \in A}^*
\end{equation}
\begin{equation}\label{me}
((\phi_{\g\d}))_{\g,\d \in A}((\J(f)_{\a\b}))_{\a\b}=((\J(f)_{\a\b}))_{\a,\b \in A}((\phi_{\g\d}))_{\g,\d \in A}.
\end{equation}
The proof of the forward direction then easily follows from the following claims.


\textbf{Claim 1.}
Let $\a,\b \in A$, with $\a=(\a_1,\hdots,\a_d)$ and $\b=(\b_1,\hdots,\b_d)$. For $z \in \Z$, we have that
\begin{equation}
  \phi_{\a\b}(0,z'')=\left\{
  \begin{array}{@{}ll@{}}
    {\a \choose \a-\b}\phi_{(\a-\b)0}(0,z'') & \text{if}\,\,\, \a_t \geq \b_t \,\,\,\forall\,\, t=1,\cdots,d, \\
    0 & \text{otherwise}.
  \end{array}\right.
\end{equation}

We note from the equation $\eqref{me}$ that, for fixed $\a,\b \in A$ and $f(z',z'')=z_1^{\b_1}\cdots z_d^{\b_d}$,  
\Bea\sum_{\g \in A}\phi_{\a\g}\J(z_1^{\b_1}\cdots z_d^{\b_d})_{\g 0} &=& \sum_{\g \in A} \J(z_1^{\b_1}\cdots z_d^{\b_d})_{\a\g}\phi_{\g 0}\Eea
on $\Z$. If $\a_t \geq \b_t$, for $t=1,\hdots,d$, only non-zero entry on the left hand side occurs for $\g =\b$ and on the right for $\g=\a-\b$. On the other hand, while $\a_t < \b_t$ for some $t \in \{1,\hdots,d\}$ the right hand side vanishes for all $\g_t \geq 0$. This verifies the Claim 1.


Thus Claim 1 shows that the matrix $((\phi_{\a\b}(0,z'')))_{\a,\b \in A}$ is a lower triangular matrix. Consequently, we have that $\Phi$ induces bundle morphisms $\Phi|_{J^{(l)}E|_{\Z}}:J^{(l)}E|_{\Z}\ra J^{(l)}\tilde{E}|_{\Z}$, for $0 \leq l \leq k$.


\textbf{Claim 2.} $\phi_{00}$ is a constant function and $\phi_{\a\a}=\phi_{00}$, for $\a \in A$, on $\Z$.


Following Claim 1 it is enough to show that $\phi_{00}$ is a constant function on $\Z$. In fact, from the equation $\eqref{jme}$ we have that
$$\<s(0,z''), s(0,z'')\>=\phi_{00}(0,z'')\<\tilde{s}(0,z''),\tilde{s}(0,z'')\>\ov{\phi_{00}(0,z'')}.$$
Consequently, Lemma $\ref{lem1}$ and Proposition $\ref{prop1}$ together yield that $$\phi_{00}(0,z'')\ov{\phi_{00}(0,0)}=1$$ completing the proof of Claim 2.


\textbf{Claim 3.} $((\phi_{\a\b}(0,z'')))_{\a,\b \in A} = \phi_{00}\cdot I$ where $I$ is the $|A|\times |A|$ identity matrix where $|A|$ is the cardinality of the set $A$.


In view of Claim 1 and Claim 2, it is enough to show that $\phi_{\a0}=0$, for $\a \in A$, on $\Z$. So calculating $\a0$-th entry of the matrices in the equation $\eqref{jme}$ and using the Lemma $\ref{lem1}$ we have
$$\langle\del^{\a} s(0,z''), s(0,w'')\rangle=\left(\sum_{\g \in A} \phi_{\a\g}(0,z'')\langle\del^{\g}\tilde{s}(0,z''), \tilde{s}(0,w'')\rangle\right)\ov{\phi}_{00}.$$ Consequently, after putting $w''=0$ and applying the Proposition $\ref{prop1}$ to the frames $\{s\}$ and $\{\tilde{s}\}$ at the origin we get $\phi_{\a0}(0,z'')=0$ on $\Z$.

\noindent Thus Claim 1, Claim 2, Claim 3 and the equation $\eqref{jme}$ together yield that
\beq\label{jec} \del^{\a}\dbar^{\b}\norm{s(0,z'')}^2=\phi_{00}\del^{\a}\dbar^{\b}\norm{\tilde{s}(0,z'')}^2 \ov{\phi}_{00}=\del^{\a}\dbar^{\b}\norm{\tilde{s}(0,z'')}^2\eeq on $\Z$, for $\a,\b \in A$.


The converse statement is easy to see. Indeed, if the equation $\eqref{jec}$ happens to be true then the desired jet bundle isomorphism $\Phi$ is given by the constant matrix $I$ \w the frames $\{\del^{\a} s\}_{\a \in A}$ and $\{\del^{\a} \tilde{s}\}_{\a \in A}$ where $I$ is the identity matrix of order $|A|$.
\end{proof}


\begin{thm}\label{jean1}
Let $\O,\Z$ be as above. Assume that both $\HM$ and $\tilde{\HM}$ contain $\C[z_1,\hdots,z_m]$ and are in $\mathrm B_r(\O)$. Then $\HM_q$ and $\tilde{\HM}_q$ are unitarily equivalent as modules over $\A(\O)$ if and only if there exists a constant unitary matrix $D$ such that  $\del^{\a}\dbar^{\b} H = D(\del^{\a}\dbar^{\b}\tilde{H})D^* $ on $\Z$, for all $\a,\b \in A=\{\a \in (\N\cup\{0\})^d:|\a|< k\}$ where $H(z)$ and $\tilde{H}(z)$ are the Gramian matrices for the holomorphic frames $\textbf{s}$ and $\tilde{\textbf{s}}$ of the {\h s} $E$ and $\tilde{E}$ on $\O$ associated to the Hilbert modules $\HM$ and $\tilde{\HM}$, respectively, normalized at origin along $\Z$.
\end{thm}

\begin{proof}
To begin with, let $\Phi:J^{(k)} E|_{\Z}\ra J^{(k)}\tilde{E}|_{\Z}$ be a jet bundle isomorphism. Then $\Phi$ can be represented by a $|A| \times |A|$ block matrix $((\Phi_{\a\b}))_{\a,\b \in A}$ \w the frames $\{\del^{\a} \f(0,z'')\}_{\a \in A}$ and $\{\del^{\a}\tilde{\f}(0,z'')\}_{\a \in A}$ where $\Phi_{\a\b}$ are holomorphic $r \times r$ matrix valued functions on $\Z$ and $|A|$ is the cardinality of $A$. The fact that $\Phi$ is an isometry of two jet bundles $J^{(k)} E|_{\Z}$ and $J^{(k)}\tilde{E}|_{\Z}$ translates to the following matrix equation on $\Z$:
\begin{equation}\label{jme1}
((\del^{\a}\dbar^{\b} H))_{\a,\b \in A} = ((\Phi_{\a\b}))_{\a,\b \in A}((\del^{\a}\dbar^{\b}\tilde{ H}))_{\a,\b \in A}(((\Phi_{\a\b}))_{\a,\b \in A})^*.
\end{equation}


\noindent Let $E_i|_{\Z}$ and $\tilde{E}_i|_{\Z}$ be the line bundles determined by the frames $\{s_i\}$ and $\{\tilde{s}_i\}$, respectively, on $\Z$, for $1 \leq i \leq r$. Note that the decompositions $E|_{\Z}=\oplus_{i=1}^r E_i|_{\Z}$ and $J^{(k)}E|_{\Z}=\oplus_{i=1}^rJ^{(k)}E_i|_{\Z}$ with $\{\del^{\a} s_i\}_{\a \in A}$ as a frame on $\Z$ are evident. Also, let $P_i:J^{(k)}E|_{\Z}\ra J^{(k)}E_i|_{\Z}$ and $\tilde{P}_i:J^{(k)}\tilde{E}|_{\Z}\ra J^{(k)}\tilde{E}_i|_{\Z}$ be the projection morphisms where the frame $\{\del^{\a} \tilde{s}_i\}_{\a \in A}$ defines the jet bundle $J^{(k)}\tilde{E}_i|_{\Z}$. It is clear that the matrix of $\Phi$ \w the frames $J(\f)=\{\sum_{\a \in A}\del^{\a} s_1\otimes \e_{\a},\hdots,\sum_{\a \in A}\del^{\a} s_r\otimes \e_{\a}\}$ and $J(\tilde{\f})=\{\sum_{\a}\del^{\a} \tilde{s}_1\otimes \e_{\a},\hdots,\sum_{\a}\del^{\a} \tilde{s}_r\otimes \e_{\a}\}$ is $(([P_{ij}]))_{i,j=1}^r$ where $[P_{ij}]$ represents the matrix of $\tilde{P}_i\Phi P^*_j$ \w the frames $\{\del^{\a} s_j\}_{\a \in A}$ and $\{\del^{\a} \tilde{s}_i\}_{\a \in A}$. Since $\Phi$ is a jet bundle isomorphism (Definition $\ref{jbi}$) it intertwines the module action on the class of holomorphic sections of $J^{(k)} E|_{\Z}$ and $J^{(k)}\tilde{E}|_{\Z}$. As a consequence, we have $$(([P_{ij}]))_{i,j=1}^r(\J(f)\otimes I_r) =  (\J(f)\otimes I_r)(([P_{ij}]))_{i,j=1}^r,$$ for $f \in \A(\O)$, which is equivalent to the fact that the bundle morphisms $\tilde{P}_i\Phi P^*_j$ intertwine the module action $\eqref{modacsec}$ on holomorphic sections of $J^{(k)} E_j|_{\Z}$ and $J^{(k)}\tilde{E}_i|_{\Z}$. Thus, $\tilde{P}_i\Phi P^*_j$ defines a jet bundle morphism from $J^{(k)}E_j|_{\Z}$ onto $J^{(k)}\tilde{E}_i|_{\Z}$.

We, therefore, can apply Claim 1 in Theorem $\ref{jean1}$ to $\tilde{P}_i\Phi P^*_j$ to conclude, for $\a,\b \in A$, that
\Bea
[P_{ij}]_{\a\b} &=& {\a \choose \a-\b}[P_{ij}]_{(\a-\b)0}(0,z''), ~(\a - \b)\in (\N \cup \{0\})^d\\
&=& {\a \choose \a-\b}(\Phi_{(\a-\b)0}(0,z''))_{ij}, ~(\a - \b)\in (\N \cup \{0\})^d,
\Eea otherwise, $[P_{ij}]_{\a\b}$ is the zero matrix.
It follows that the matrix of $\Phi(0,z'')$ \w the frames $\{\del^{\a}\f\}_{\a \in A}$ and $\{\del^{\a}\tilde{\f}\}_{\a \in A}$ is a lower triangular block matrix with $\Phi_{\a\a}(0,z'')=\Phi_{00}(0,z'')$ for $\a \in A$ and, for $\a,\b \in A$, $\a=(\a_1,\hdots,\a_d),\b=(\b_1,\hdots,\b_d)$ and $1 \leq i,j \leq r$,
\beq\label{mb}
(\Phi_{\a\b}(0,z''))_{ij}={\a \choose \a-\b}(\Phi_{(\a-\b)0}(0,z''))_{ij} \text{  if } (\a - \b)\in (\N \cup \{0\})^d,
\eeq
and is zero, otherwise, on $\Z$. A similar proof as in Claim 2 in Theorem $\ref{jean}$ with matrix valued holomorphic functions $H, \tilde{H}$ and $\Phi_{00}$ on $\Z$ yields that $\Phi_{00}$ is a constant unitary matrix. Thus the proof will be done once we prove that $\Phi_{\a0}=0$, for $\a \in A$, on $\Z$. Computing the $\a0$-th block of the matrices in the equation $\eqref{jme1}$ and using the Lemma $\ref{lem1}$ we get
$$((\langle\del^{\a} s_i(0,z''), s_j(0,w'')\rangle))_{i,j=1}^r=\left(\sum_{t=0}^{\a} \Phi_{\a\b}(0,z'')((\langle\del^{\b}\tilde{s}_i(0,z''), \tilde{s}_j(0,w'')\rangle))_{i,j=1}^r\right)\Phi^*_{00}.$$ Consequently, after putting $w''=0$ and applying Proposition $\ref{prop1}$ to the frames $\f$ and $\tilde{\f}$ at the origin we get $\Phi_{\a0}(0,z'')=0$ on $\Z$. Thereby from $\eqref{jme1}$ we have
\beq\label{jec1}\del^{\a}\dbar^{\b} H(0,z'') = D(\del^{\a}\dbar^{\b} \tilde{H}(0,z''))D^*\eeq on $\Z$ for all $\a,\b \in A$ where $D=\Phi_{00}$.


For the converse direction, note that the equation $\eqref{jec1}$ canonically gives rise to the jet bundle isomorphism $\Phi$ by prescribing the matrix of $\Phi$ as $D\otimes I$ \w the frames $J(\f)$ and $J(\tilde{\f})$ where $I$ is the identity matrix of order $|A|$.
\end{proof}

\begin{cor}
Let $\textbf{T}=(T_1,\hdots,T_m)$ and $\tilde{\textbf{T}}=(\tilde{T}_1,\hdots,\tilde{T}_m)$ be two $m$-tuples of operators in $\mathrm B_1(\O)$. Then $\textbf{T}$ and $\tilde{\textbf{T}}$ are unitarily equivalent if and only if there are jet bundle isomorphisms $\Phi_k:J^{(k)} E|_{\Z}\ra J^{(k)}\tilde{E}|_{\Z}$, for every $k \in \N \cup \{0\}$ where $\Z$ is any singleton set $\{p\}$, for $p \in \O$.
\end{cor}

\begin{proof}
The necessity part is trivial and so we only show that $\textbf{T}$ and $\tilde{\textbf{T}}$ are unitarily equivalent assuming that there are jet bundle isomorphisms $\Phi_k:J^k E|_{\Z}\ra J^k\tilde{E}|_{\Z}$, for every $k \in \N \cup \{0\}$.


Let $E$ and $\tilde{E}$ be vector bundles over $\O$ corresponding to operator tuples $\textbf{T}$ and $\tilde{\textbf{T}}$, respectively, and $\Z=\{0\}$. Note that the codimension of $\Z$ is $m$. Let $s$ and $\tilde{s}$ be frames for $E$ and $\tilde{E}$, respectively, normalized at origin. Using Theorem $\ref{jean}$ we have, for every $k \in \N \cup \{0\}$, that 
\beq\label{tc}\del^{\a_1}_1\cdots\del^{\a_m}_m\dbar^{\b_1}_1\cdots\dbar^{\b_m}_m\norm{s(0)}^2 &=&
\del^{\a_1}_1\cdots\del^{\a_m}_m\dbar^{\b_1}_1\cdots\dbar^{\b_m}_m\norm{\tilde{s}(0)}^2\eeq
 for all $\a,\b \in A(k)$ where $A(k)= \{\a\in (\N\cup\{0\})^d:|\a|<k\}$.


Since $s$ and $\tilde{s}$ both are holomorphic on their domains of definition, $\norm{s}^2$ and $\norm{\tilde{s}}^2$ are real analytic there. Consequently, using the power series expansion of $\norm{s}^2$ and $\norm{\tilde{s}}^2$ together with the equation $\eqref{tc}$ we obtain that $$\norm{s(z)}^2=\norm{\tilde{s}(z)}^2$$ on some open {\nbhd}, say $\O_0$, of the origin in $\O$. Thus the bundle map $\Phi:E\ra \tilde{E}$ determined by the formula $\Phi(s(z))=\tilde{s}(z)$ defines an isometric bundle isomorphism between $E$ and $\tilde{E}$ over $\O_0$. Then the result is a direct consequence of the Rigidity theorem in \cite{CGOT}.
\end{proof}

\begin{rem}
Note that the theorem above shows that the unitary equivalence of local operators (1.5 in \cite{CGOT}) $N^{(k)}_{\o_0}$ and $\tilde{N}^{(k)}_{\o_0}$ corresponding to $\textbf{T}$ and $\tilde{\textbf{T}}$, respectively, for all $k\geq 0$ but at a fixed point $\o_0 \in \O$ implies the unitary equivalence of $\textbf{T}$ and $\tilde{\textbf{T}}$. In other words, any $m$-tuples of operators $\textbf{T}\in \mathrm B_1(\O)$ enjoy the "Taylor series expansion" property. Moreover, following the technique used in Theorem 18 in \cite{ALOTCD}, it is seen that the same property is also enjoyed by any $\textbf{T}\in \mathrm B_r(\O)$, $r \geq 1$.
\end{rem}

The following theorem is one of the main results in this article which generalizes the study of quotient modules done in the paper \cite{EQHMII} to the case of arbitrary codimension. For the definition of bundle maps used in the following theorem, we refer the readers to the equation $\eqref{covd1}$ in Section \ref{prelim}.

\begin{thm}\label{mthm}
Let $\O \subset \C^m$ be a bounded domain containing the origin and $\Z \subset \O$ be the complex manifold of codimension $d$ defined by $z_1=\cdots=z_d=0$. Suppose that pair of Hilbert modules $\HM$ and $\tilde{\HM}$ are in $\mathrm B_r(\O)$ and contain $\C[z_1,\hdots,z_m]$. Then $\HM_q$ and $\tilde{\HM}_q$ are isomorphic as modules over $\A(\O)$ if and only if following conditions hold:
\begin{enumerate}
\item[(i)] There exists holomorphic isometric bundle map $\Phi:E|_{\Z}\ra \tilde{E}|_{\Z}$ where $E$ and $\tilde{E}$ are {\h s} over $\O$ corresponding to the Hilbert modules $\HM$ and $\tilde{\HM}$ over $\A(\O)$.
\item[(ii)] The transverse curvature $\dbar_i((\del_j H) H^{-1})$ and $\dbar_i((\del_j\tilde{H})\tilde{H}^{-1})$ of $E$ and $\tilde{E}$, respectively, for $1\leq i,j\leq d$, as well as their covariant derivatives of order at most $k-2$, along the transverse directions to $\Z$, are intertwined by $\Phi$ on $\Z$ where $\textbf{s}=\{s_1,\hdots,s_r\}$ and $\tilde{\textbf{s}}=\{\tilde{s}_1,\hdots,\tilde{s}_r\}$ are frames of $E$ and $\tilde{E}$ normalized at origin along $\Z$, respectively, for $\a \in A$, and $H$ and $\tilde{H}$ are Gramians of $\textbf{s}$ and $\tilde{\textbf{s}}$, respectively.
\item[(iii)] $\Phi$ intertwines the bundle maps $\J_i^{\a}(H):=\dbar_i((\del^{\a}H)H^{-1})$ and $\J_i^{\a}(\tilde{H}):=\dbar_i((\del^{\a} \tilde{H})\tilde{H}^{-1})$, for $d+1 \leq i \leq m$ and $\a \in A$ on $\Z$ where $\textbf{s}=\{s_1,\hdots,s_r\}$ and $\tilde{\textbf{s}}=\{\tilde{s}_1,\hdots,\tilde{s}_r\}$ are frames of $E$ and $\tilde{E}$ normalized at origin along $\Z$, respectively, $H$ and $\tilde{H}$ are Gramians of $\textbf{s}$ and $\tilde{\textbf{s}}$, respectively.
\end{enumerate}
\end{thm}

\begin{rem}\label{remofmainthm}
Although it may seem apparently that the condition (iii) in the theorem above depends on the choice of a frame, it is not the case. For instance, if $\textbf{\textit{t}}$ is another frame normalized at origin along $\Z$ we have $\textbf{\textit{t}}=\textbf{\textit{s}}X$ for some holomorphic function $X:\Z\ra GL_r(\C)$. Since both $\textbf{\textit{s}}$ and $\textbf{\textit{t}}$ are normalized at origin along $\Z$ the same proof as in Claim 2 in Theorem $\ref{jean}$ with matrix valued holomorphic functions shows that $X$ is a constant unitary matrix. Thus we have $H=XHX^*$ and hence it follows that $$\J_i^{\a}(G)=X\J_i^{\a}(H)X^{-1}, d+1 \leq i \leq m,$$ where $G$ is the Gramian matrix of the frame $\textbf{\textit{t}}$.
\end{rem}

\begin{proof}
Let $\O \subset \C^m$ and $\Z \subset \O$ be as given. Suppose that $\HM_q$ and $\tilde{\HM}_q$ are equivalent as modules over $\A(\O)$. Then by Theorem $\ref{jean1}$ there exists a constant unitary matrix $D$ such that \beq\label{dmeq}\del^{\a}\dbar^{\b} H(0,z'')=D(\del^{\a}\dbar^{\b}\tilde{H}(0,z''))D^*,\text{ for }(0,z'')\in \Z \text{ and }\a,\b \in A,\eeq where $H(z)$ and $\tilde{H}(z)$ are the Gramian matrices for holomorphic frames $\textbf{\textit{s}}=\{s_1,\hdots,s_r\}$ and $\tilde{\textbf{\textit{s}}}=\{\tilde{s}_1,\hdots,\tilde{s}_r\}$ for $E$ and $\tilde{E}$ on $\O$ associated to the Hilbert modules $\HM$ and $\tilde{\HM}$, respectively, normalized at origin along $\Z$. In particular, for $\a=\b=0$, $\eqref{dmeq}$ becomes $$H(0,z'')=D\tilde{H}(0,z'')D^*, \text{ for }(0,z'')\in \Z.$$


Let $\Phi:E|_{\Z}\ra\tilde{E}|_{\Z}$ be the bundle morphism whose matrix representation \w the frames $\textbf{\textit{s}}$ and $\tilde{\textbf{\textit{s}}}$ is $D$ providing the desired isometric bundle map in (i). The equation $\eqref{dmeq}$ together with (i) of Lemma $\ref{cume}$ yield (ii), and since $D$ is a constant unitary matrix on $\Z$, (iii) is an easy consequence of $\eqref{dmeq}$ with $\b=0$.


For the converse direction, we show that the condition (i), (ii) and (iii) in the statement together imply the condition of Theorem $\ref{jean1}$. More precisely, we prove that there exists a constant unitary matrix $D$ on $\Z$ such that the equation $\eqref{dmeq}$ holds on the submanifold $\Z$ for $\a,\b \in A$ and frames $\textbf{\textit{s}}$ and $\tilde{\textbf{\textit{s}}}$ of $E$ and $\tilde{E}$, respectively, normalized at origin along $\Z$. We follow two steps in obtaining the matrix $D$. First, we extend the holomorphic isometric bundle map $\Phi: E|_{\Z} \ra \tilde{E}|_{\Z}$, obtained from condition (i), to a family of linear isometries $\hat{\Phi}_{z_0}: J^{(k)}E|_{z_0} \ra J^{(k)}\tilde{E}|_{z_0}$ for every $z_0 \in \Z$. Then this extension is shown to be a jet bundle isomorphism providing our desired matrix. 


Let us begin with frames $\textbf{\textit{s}}$ and $\tilde{\textbf{\textit{s}}}$ for $E$ and $\tilde{E}$, respectively, normalized at $z_0 \in \Z$, for an arbitrary $z_0 \in \Z$. Condition (i) yields an isometric holomorphic bundle map $\Phi: E|_{\Z} \ra \tilde{E}|_{\Z}$ and consequently, we have a holomorphic $r \times r$ matrix valued function $\phi$ on $\Z$ such that \beq\label{meq} H(0,z'') &=& \phi(0,z'')\tilde{H}(0,z'')\phi(0,z'')^*\eeq where $\phi$ represents $\Phi$ \w frames $\f$ and $\tilde{\f}$. Since both $\f$ and $\tilde{\f}$ are normalized at $z_0$, the equation $\eqref{meq}$ shows that $\phi(0,z''_0)$ is a unitary matrix. Furthermore, from condition (ii) of our hypothesis along with second statement of Lemma $\ref{cume}$ we have, for $0 \leq \a_1+\cdots+\a_d \leq k-1,0 \leq \b_1+\cdots+\b_d \leq k-1$, 
\beq\label{jmeq}\del_1^{\a_1}\cdots\del_d^{\a_d}\dbar_1^{\b_1}\cdots\dbar_d^{\b_d}
H(0,z''_0) = \phi(0,z''_0)\del_1^{\a_1}\cdots\del_d^{\a_d}\dbar_1^{\b_1}\cdots
\dbar_d^{\b_d}\tilde{H}(0,z''_0)\phi(0,z''_0)^*\eeq as $\del_1^{\a_1}\cdots\del_d^{\a_d}H(0,z''_0)$ (respectively, $\del_1^{\a_1}\cdots\del_d^{\a_d}\tilde{H}(0,z''_0)$) and $\dbar_1^{\b_1}\cdots
\dbar_d^{\b_d}H(0,z''_0)$ (respectively, $\dbar_1^{\b_1}\cdots
\dbar_d^{\b_d}\tilde{H}(0,z''_0)$) are zero matrices for any $\a_i, \b_i \geq 0$, $i=1,\hdots,d$. Thus the equations above $(\ref{meq},~\ref{jmeq})$ lead to the following natural isometric extension, $\hat{\Phi}_{z_0}: J^{(k)}E|_{z_0} \ra J^{(k)}\tilde{E}|_{z_0}$ defined by
\beq\label{ext} \hat{\Phi}_{z_0}(\del^{\a}s_j(0,z''_0))=\sum_{i=1}^r\phi_{ji}(0,z''_0)\del^{\a}\tilde{s}_i(0,z''_0), \text{ } \a \in A, \text{ } 1 \leq j \leq r.\eeq
We note, for $\a \in A$, $ 1 \leq j \leq r$, $z_0=(0,z''_0)\in \Z$ and $f \in \A(\O)$, that
$$\hat{\Phi}_{z_0}\J(f)(z_0)(\del^{\a} s_j(0,z''_0)) = \hat{\Phi}_{z_0}(\sum_{\b}\del^{\a-\b}f(z_0)\del^{\b} s_j(0,z''_0)) = \J(f)(z_0)(\hat{\Phi}_{z_0}(\del^{\a} s_j(0,z''_0))$$
implying that the extension $\eqref{ext}$ above intertwines the module action $\eqref{modacsec}$ on the sections of $J^{(k)}E$ and $J^{(k)} \tilde{E}$ over $\Z$. From now on, in the rest of the proof, we denote this extension by $\hat{\Phi}$. We also note that $\hat{\Phi}$ is an isometry on $J^{(k)}E|_{\Z}\ra \Z$.


Consider frames $\f$ and $\tilde{\f}$ of $E$ and $\tilde{E}$, respectively, normalized at origin along $\Z$, and $D(0,z''):=((\hat{\Phi}_{\a\b}(0,z'')))_{\a,\b \in A}$ be the matrix of $\hat{\Phi}$ \w the frames $\{\del^{\a} \f(0,z'')\}_{\a \in A}$ and $\{\del^{\a}\tilde{\f}(0,z'')\}_{\a \in A}$. We point out that the $(0,0)$-th block of $D(0,z'')$, namely, $\hat{\Phi}_{00}(0,z'')$ is the matrix representation of $\Phi: E|_{\Z} \ra \tilde{E}|_{\Z}$ \w $\f$ and $\tilde{\f}$, and hence $\hat{\Phi}_{00}(0,z'')$ is holomorphic on $\Z$. So in view of the proof of Claim 2 in Theorem $\ref{jean}$ with matrix valued holomorphic functions, it can be seen that $\hat{\Phi}_{00}(0,z'')$ is a constant unitary matrix, say, $\hat{\Phi}_{00}$ on $\Z$.


We also note, from the construction of $\hat{\Phi}$ above, that $\hat{\Phi}(J^{(l)}E|_{\Z}) \subset J^{(l)}\tilde{E}|_{\Z}$, for $0 \leq l \leq k$. Consequently, $D(0,z'')$ is a lower triangular matrix. Moreover, since $\hat{\Phi}$ commutes with the module action on the sections of $J^{(k)} E|_{\Z}$ and $J^{(k)}\tilde{E}|_{\Z}$, the same proof as in Theorem $\ref{jean1}$ shows that the entries of the matrix satisfy the properties stated in $\eqref{mb}$. So it is enough to show that $\hat{\Phi}_{\a0}(0,z'')=0$ on $\Z$ for $\a \in A$. We prove this with the help of mathematical induction on $|\a|$.


Since $\hat{\Phi}$ is an isometry on $J^{(k)}E|_{\Z}\ra \Z$ it satisfies the equation $\eqref{jme1}$ which leads to 
\Bea\del^{\a}H(0,z'') &=& ((\Phi_{\g\d}(0,z'')))_{\g,\d \in A}\left(\sum_{\tau \in A}\del^{\si}\dbar^{\tau}\tilde{H}(0,z'')\Phi_{0\tau}(0,z'')^*\right)_{\si \in A}\\
&=& \sum_{\si \in A}\hat{\Phi}_{\a\si}(0,z'')\del^{\a}\tilde{H}(0,z'')\hat{\Phi}_{00}(0,z'') 
\Eea
with $\hat{\Phi}_{00}$ being a constant unitary matrix. Here the second equality holds because the matrix $((\hat{\Phi}_{\a\b}(0,z'')))_{\a,\b \in A}$ is lower triangular. Assume that $\hat{\Phi}_{\si0}(0,z'')=0$ on $\Z$ for all multi-indices $\si=(\si_1,\hdots,\si_d)$ with $0 < |\si| < |\a|$. For $|\a|=1$ this assumption is empty and hence automatically fulfilled. We then have on $\Z$ that 
\beq
\nonumber \del^{\a}H(0,z'')\hat{\Phi}_{00}(0,z'') &=& \sum_{\si \leq \a}\hat{\Phi}_{\a\si}(0,z'')\del^{\si}\tilde{H}(0,z'')\\
\nonumber &=&\hat{\Phi}_{\si0}(0,z'')\tilde{H}(0,z'')+\hat{\Phi}_{\a\a}(0,z'')\del^{\a}\tilde{H}(0,z'')\\
\nonumber &+& \sum_{0 < \si < \a}\hat{\Phi}_{\a\si}(0,z'')\del^{\si}\tilde{H}(0,z'').
\eeq 
Since $\hat{\Phi}$ intertwines the module action the equation $\eqref{mb}$ yields that $\hat{\Phi}_{\a\si}(0,z'')=\hat{\Phi}_{(\a-\si)0}(0,z'')$ on $\Z$ whenever $0 \leq \si \leq \a$. Therefore, with the help of induction hypothesis we have $\hat{\Phi}_{\a\si}(0,z'')=0$ provided $0 < \si < \a$. Also, it follows from the above calculation that   
\Bea
\hat{\Phi}_{\a0}(0,z'')&=&\del^{\a}H(0,z'')\hat{\Phi}_{00}(0,z'') \tilde{H}(0,z'')^{-1}-\hat{\Phi}_{00}(0,z'')\del^{\a}\tilde{H}(0,z'')\tilde{H}(0,z'')^{-1}\\
&=& \del^{\a}H(0,z'')\tilde{H}(0,z'')^{-1}\hat{\Phi}_{00}(0,z'') -\hat{\Phi}_{00}(0,z'')\del^{\a}\tilde{H}(0,z'')\tilde{H}(0,z'')^{-1}.
\Eea
Now, for $d+1 \leq i \leq m$, we get
$$\dbar_i\hat{\Phi}_{\a0}(0,z'')= \dbar_i(\del^{\a}H(0,z'')\tilde{H}(0,z'')^{-1})\hat{\Phi}_{00} -\hat{\Phi}_{00}\dbar_i(\del^{\a}\tilde{H}(0,z'')\tilde{H}(0,z'')^{-1})= 0$$
as $\hat{\Phi}_{00}$ is constant. Thus, $\hat{\Phi}_{\a0}(0,z'')$ is holomorphic on $\Z$. Consequently, from the equation $\eqref{mb}$ it follows that $\hat{\Phi}$ is a holomorphic bundle morphism. Now the identity obtained above 
\Bea((\<\del^{\a}s_i(0,z''),s_j(0,z'')\>))_{i,j=1}^r\hat{\Phi}_{00} &=& \hat{\Phi}_{\a0}(0,z'')((\<\tilde{s}_i(0,z''),\tilde{s}_j(0,z'')\>))_{i,j=1}^r\\
&+& \hat{\Phi}_{00}((\<\del^{\a}\tilde{s}_i(0,z''),\tilde{s}_j(0,z'')\>))_{i,j=1}^r\Eea can be polarised using Lemma $\ref{lem1}$ to obtain 
\Bea((\<\del^{\a}s_i(0,z''),s_j(0,0)\>))_{i,j=1}^r\hat{\Phi}_{00} &=& \hat{\Phi}_{\a0}(0,z'')((\<\tilde{s}_i(0,z''),\tilde{s}_j(0,0)\>))_{i,j=1}^r\\
&+& \hat{\Phi}_{00}((\<\del^{\a}\tilde{s}_i(0,z''),\tilde{s}_j(0,0)\>))_{i,j=1}^r.\Eea
This implies that $\hat{\Phi}_{\a0}(0,z'')=0$ since the frames $\f$ and $\tilde{\f}$ are normalized at origin along $\Z$.
\end{proof}

\begin{rem}
(i) From the proof of Theorem $\ref{mthm}$, it is clear that, for $r=1$ and $d=k=2$, conditions (i), (ii) and (iii) of Theorem $\ref{mthm}$ together yield that the curvatures of the bundles $E|_{\Z}$ and $\tilde{E}|_{\Z}$ are equal. Further, the matrix in $\eqref{msi}$ turns out to be the diagonal matrix $\psi_{00}I$ \w a normalized frame at origin where $I$ is the identity matrix of order $3$. Moreover, following the proof of Claim 2 in Theorem $\ref{jean}$ we see that $\psi_{00}$ is a constant function on $\Z$ with $|\psi_{00}|=1$. Thus Theorem $\ref{mthm}$ is exact generalization of Lemma $\ref{trial}$.


(ii) We also note that three conditions $(i),(ii)$ and $(iii)$ listed in the theorem above  correspond to the conditions that the metric of $E$ and $\tilde{E}$ are equivalent to order $k$, in the sense of the paper \cite{EQHMII}, on $\Z$ while the codimension of $\Z$ is $1$. Consequently, following \cite[Remark 6.1]{EQHMII}, we see that the conditions (i), (ii) and (iii) in the above theorem correspond to the equality on $\Z$ of tangential curvatures, transversal curvatures and the second fundamental forms for the inclusions $E|_{\Z} \subset J^{(2)}E|_{\Z}$ and $\tilde{E}|_{\Z}\subset J^{(2)}\tilde{E}|_{\Z}$, or equivalently, the off-diagonal entries of the curvature matrices of the bundles $E$ and $\tilde{E}$, for $k=2$.
\end{rem}

\section{Application}\label{exap}

Let us consider the family of Hilbert modules $\text{Mod}(\D^m):=\{\H^{(\l)}(\D^m):\l=(\l_1,\hdots,\l_m),\l_j>0,~1\leq j\leq m\}$ over the polydisc $\D^m$ in $\C^m$. We now prove that for any pair of tuples $\l=(\l_1,\hdots,\l_m)$ and $\l'=(\l_1',\hdots,\l_m')$, the unitary equivalence of two quotient modules $\H^{(\l)}_q$ and $\H^{(\l')}_q$, obtained from the submodules of functions vanishing of order $2$ along the diagonal set $\Delta$, implies the equality of the Hilbert modules $\H^{(\l)}$ and $\H^{(\l')}$. In other words, the restriction of the curvature of the jet bundle $J^{(2)}E^{(\l)}$ to the diagonal $\Delta$ is a complete unitary invariant for the class $\text{Mod}(\D^m)$ where the jet bundle $J^{(2)}E^{(\l)}$ is defined by the global frame $\{K^{(\l)}(.,\ov{w}),\del_1 K^{(\l)}(.,\ov{w}),\hdots,\del_m K^{(\l)}(.,\ov{w})\}$ where $\del_j$ are the differential operators \w the variable $z_j$, for $j=1,\hdots,m$

\begin{thm}\label{WHM}
For $\l=(\l_1,\hdots,\l_n)$ and ${\l}'=({\l}'_1,\hdots,{\l}'_n)$ with $\l_i,{\l}'_i >0$, for all $i=1,\hdots,n$, the quotient modules $\H^{(\l)}_q$ and $\H^{({\l}')}_q$ are unitarily equivalent if and only if $\l_i={\l}'_i$, for all $i=1,\hdots,n$.
\end{thm}

\begin{proof}
The proof of sufficiency is trivial. So we only prove the necessity. Let us begin by pointing out that the diagonal set $\Delta$ in $\D^m$ can be described as the zero set of the ideal $I:=<z_1-z_2,\hdots,z_i-z_{i+1},\hdots,z_{m-1}-z_m>$. It is easy to verify that $\phi: U \ra \C^m$ defined by $$\phi(z_1,\hdots,z_m)=(z_1-z_2,\hdots,z_i-z_{i+1},\hdots,z_{m-1}-z_m,z_m)$$ yields an admissible coordinate system (Definition $\ref{ad}$) around the origin. We choose $U$ small enough so that $\phi(U)\subset \D^m$. A simple calculation shows that $\phi^{-1}:\phi(U)\ra U$ takes the form $$\phi^{-1}(u_1,\hdots,u_m)=(\sum_{j=1}^m u_j, \hdots,\sum_{j=i}^m u_j,\hdots,u_{m-1}+u_m, u_m).$$

\noindent For rest of the proof we pretend $U$ to be $\D^m$ following Remark $\ref{remres}$. Using Proposition $\ref{eococh}$ it is enough to prove that ${\l}_i={\l}'_i$, $i=1,\hdots,m$, provided $\phi^*\H^{(\l)}_q$ is unitarily equivalent to $\phi^*\H^{({\l}')}_q$ where $\phi^*\H^{(\l)}_q$ and $\phi^*\H^{({\l}')}_q$ are the quotient modules obtained from the submodules $\phi^*\H^{(\l)}_0$ and $\phi^*\H^{({\l}')}_0$ of the Hilbert modules $\phi^*\H^{(\l)}$ and $\phi^*\H^{({\l}')}$, respectively.


Note that both $\phi^*\H^{(\l)}$ and $\phi^*\H^{({\l}')}$ are reproducing kernel Hilbert modules with reproducing kernels
$$\mathsf{K}(\textbf{u})=\prod_{i=1}^m\left(1-|{\sum_{j=1}^m u_j}|^2\right)^{-{\l}_i}\,\,\,\text{and}\,\,\,\mathsf{K'}(\textbf{u})=\prod_{i=1}^m\left(1-|{\sum_{j=1}^m u_j}|^2\right)^{-{\l}'_i},$$
respectively, where $\textbf{u}=(u_1,\hdots,u_m)\in \phi(U)$. We also pint out that the submodules $\phi^*\H^{(\l)}_0$ and $\phi^*\H^{({\l}')}_0$ consist of functions in $\phi^*\H^{(\l)}$ and $\phi^*\H^{({\l}')}$, respectively, vanishing along the submanifold $\Z:=\{(0,\hdots,0,u_m):u_m \in \D\}\cap \phi(U)$ of order $2$.


\noindent Since $\phi^*\H^{(\l)}_q$ and $\phi^*\H^{({\l}')}_q$ are unitarily equivalent from Theorem $\ref{mthm}$ it follows that \beq\label{eqn3}\mathcal{K}|_{\Z}&=&\mathcal{K}'|_{\Z}\eeq where $\mathcal{K}$ and $\mathcal{K}'$ are the curvature matrices for the vector bundles $E$ and $E'$ over $\phi(U)$ obtained from the Hilbert modules $\phi^*\H^{(\l)}$ and $\phi^*\H^{({\l}')}$, respectively. By the definition of curvature tensors we have that $\mathcal{K}(\textbf{u})=((\mathcal{K}_{ij}(\textbf{u})))_{i,j=1}^m$ where
$$\mathcal{K}_{ij}(\textbf{u})=\frac{\del^2}{\del u_i\del \ov{u}_j}\log\mathsf{K}(\textbf{u},\textbf{u}),$$ for $\textbf{u}=(u_1,\hdots,u_m)\in \phi(U)$. Thus, for $1 \leq i \leq m$ and $\textbf{u}\in \phi(U)$,
$$
\mathcal{K}_{ii}(\textbf{u}) = \frac{\del^2}{\del u_i\del \ov{u}_i}\log\mathsf{K}(\textbf{u},\textbf{u}) = \sum_{l=1}^i {\l}_l\left(1-|{\sum_{j=l}^m u_j}|^2\right)^{-1}.$$
A similar computation also yields, for $i=1,\hdots,m$ and $\textbf{u} \in \phi(U)$, that
$$\mathcal{K}'_{ii}(\textbf{u}) = \sum_{l=1}^i {\l}'_l\left(1-|{\sum_{j=l}^m u_j}|^2\right)^{-1}.$$
We now note that, for $\textbf{u}\in \Z$, $$\mathcal{K}_{ii}(\textbf{u})=\frac{\sum_{l=1}^i {\l}_l}{(1-|u_m|^2)}\,\,\,\text{and}\,\,\,\mathcal{K}'_{ii}(\textbf{u})=\frac{\sum_{l=1}^i {\l}'_l}{(1-|u_m|^2)}.$$
Thus by using the equality in equation $\eqref{eqn3}$ it is not hard to see that ${\l}_i={\l}'_i$, $i=1,\hdots,m$.
\end{proof}

This example shows that the unitary equivalence of two quotient modules determines the Hilbert modules completely. In general, it seems natural to ask if the unitary equivalence of certain quotient modules implies the unitary equivalence of the Hilbert modules. More precisely, let $\HM$ and $\tilde{\HM}$ be two Hilbert modules in $\mathrm B_r(\O)$ and $1\leq d\leq m$ be any integer. Let $\HM_q(\mathsf{Z})$ (respectively, $\tilde{\HM}_q(\mathsf{Z})$) be the quotient module obtained from the submodule of functions in $\HM$ (respectively, $\tilde{\HM}$) vanishing of order $t$ along a connected complex submanifold $\mathsf{Z}$ of codimension $d$. Then the question of interest is to find conditions on $t$ and $d$ such that the unitary equivalence of $\HM$ and $\tilde{\HM}$ is determined by the unitary equivalence of $\HM_q(\mathsf{Z})$ and $\tilde{\HM}_q(\mathsf{Z})$ for every $d$ codimensional connected complex submanifold $\mathsf{Z}$ in $\O$. This question was studied by Chen and Douglas in \cite{LOCHM} for $d=m$. In our case, we consider any $d$ with $1\leq d\leq m$. We first show, for $r\leq d\leq m$ and $t_0=r+2$, that the unitary equivalence of quotient modules forces the Hilbert modules to be unitarily equivalent. It is also pointed out, for $d<r$, that there exist Hilbert modules $\HM$ and $\tilde{\HM}$ which are not equivalent although the quotient modules $\HM_q(\mathsf{Z})$ and $\tilde{\HM}_q(\mathsf{Z})$ are for $\mathsf{Z}$ of codimension $d$ and $t\leq r+2$. We begin by recalling the following definition and theorem from \cite{EQCN}.

\begin{defn}
Let $E$ and $\tilde{E}$ be two hermitian holomorphic vector bundles over $\O$ and $k$ be a positive integer. Then $E$ and $\tilde{E}$ are equivalent to order $k$ at a point $w\in\O$ if there exists a linear isometry $\Phi_w:E_w\ra\tilde{E}_w$ such that $$\Phi_w\circ\chi|_w=\tilde{\chi}|_w\circ\Phi_w$$ for each covariant derivative of the curvatures $\chi|_w$ and $\tilde{\chi}|_w$ of $E$ and $\tilde{E}$ at $w$, respectively, of total order at most $k$.
\end{defn}

\begin{thm}\cite[Theorem II, 3.10]{EQCN}\label{equivalence of connections}
Let $E$ and $\tilde{E}$ be two hermitian holomorphic vector bundles over $\O$ of rank $r$ which are equivalent to order $r$ at every point in $\O$. Then $E$ and $\tilde{E}$ are locally unitarily equivalent on an open dense subset of $\O$.
\end{thm}

We now present the main theorem in this section which generalizes Theorem 1.6 in \cite{CGOT}.

\begin{thm}\label{localization}
Let $\HM$ and $\tilde{\HM}$ be two Hilbert modules in $\mathrm B_r(\O)$ with $r<m=\text{dim}~\O$ and $S$ be the collection of all connected complex submanifolds $\mathsf{Z}$ in $\O$ of codimension $d$ with $r\leq d\leq m$. Assume also that both $\HM$ and $\tilde{\HM}$ contain $\C[z_1,\hdots,z_m]$. Let $\HM_q(\mathsf{Z})$ (respectively, $\tilde{\HM}_q(\mathsf{Z})$) be the quotient module obtained from the submodule of functions in $\HM$ (respectively, in $\tilde{\HM}_q$) vanishing of order $r+2$ along the submanifold $\mathsf{Z}\in S$. If for every $\mathsf{Z}\in S$ the quotient modules $\HM_q(\mathsf{Z})$ and $\tilde{\HM}_q(\mathsf{Z})$ are unitarily equivalent as Hilbert modules over $\A(\O)$ then $\HM$ is unitarily equivalent to $\tilde{\HM}$ as Hilbert modules over $\A(\O)$.
\end{thm}

\begin{proof}
In view of \eqref{Res} in Section \ref{intro}, it is enough to show that $\HM$ and $\tilde{\HM}$ are locally equivalent which we prove with the help of Theorem \ref{equivalence of connections}. Let $w^{0}$ be an arbitrary but fixed point in $\O$, $I=(i_1,\hdots,i_d)$ with $1\leq i_1<\cdots<i_d\leq m$ and $\a=(\a_1,\hdots,\a_d)\in (\N\cup\{0\})^d$ with $0\leq |\a|\leq r$. Consider the submanifold $\mathsf{Z}\subset \O$ defined as $\mathsf{Z}:=w^0+\Z$ where $\Z$ is the coordinate plane defined by $\Z=\{z\in\O:z_{i_1}=\cdots=z_{i_d}=0\}$. Assume that $\HM_q$ (respectively, $\tilde{\HM}_q$) is the quotient module obtained from the submodule of functions in $\HM$ (respectively, $\tilde{\HM}$) vanishing of order $r+2$ along the submanifold $\mathsf{Z}$.


Let $U$ be an open neighbourhood of origin in $\O$ such that $w^0+U\subset \O$ and $\phi:U\ra\phi(U)$ be the bi-holomorphism defined by $z\mapsto z+w^0$. It follows from Proposition \ref{eococh} that $\phi^*(\HM_q)$ is unitarily equivalent to $\phi^*(\tilde{\HM}_q)$. Note that $\phi^*(\HM_q)$ (respectively, $\phi^*(\tilde{\HM}_q)$) is the quotient module obtained from the submodule of functions in $\phi^*(\HM|_{\phi(U)})$ (respectively, $\phi^*(\tilde{\HM}|_{\phi(U)})$) vanishing of order $r+2$ along the coordinate plane $\Z$. Further, since $\HM$ (respectively, $\tilde{\HM}$) is in $\mathrm B_{r}(\O)$ and $\C[z_1,\hdots,z_m]$ is contained in both $\HM$ and $\tilde{\HM}$, it follows from Theorem \ref{qmodinCDcls} that $\phi^*(\HM|_{\phi(U)})$ (respectively, $\phi^*(\tilde{\HM}|_{\phi(U)})$) is in $\mathrm B_{r}(\phi(U))$. We also point out from the definition of pull back that $\phi^*E|_{\phi(U)}\ra U$ (respectively, $\phi^*\tilde{E}|_{\phi(U)}\ra U$) are the hermitian holomorphic vector bundles associated to $\phi^*(\HM|_{\phi(U)})$ (respectively, $\phi^*(\tilde{\HM}|_{\phi(U)})$) where $E\ra \O$ and $\tilde{E}\ra \O$ are the hermitian holomorphic vector bundles associated to $\HM$ and $\tilde{\HM}$, respectively, and $\phi^*E|_{\phi(U)}\ra U$ (respectively, $\phi^*\tilde{E}|_{\phi(U)}\ra U$)  are the pull back of $E|_{\phi(U)}\ra \phi(U)$ and $\tilde{E}|_{\phi(U)}\ra \phi(U)$.


It follows from part (i) and part (ii) in Theorem \ref{mthm} that there exists a isometric bundle map $\Phi:\phi^*E|_{\Z}\ra \phi^*\tilde{E}|_{\Z}$ which intertwines the transverse curvatures $\dbar_{i_l}((\del_{i_j}\phi^*H)\phi^*H^{-1})$ and $\dbar_{i_l}((\del_{i_j}\phi^*\tilde{H})\phi^*\tilde{H}^{-1})$ of $\phi^*E$ and $\phi^*\tilde{E}$, respectively, for $1\leq j,l\leq d$, as well as their covariant derivatives along $z_{i_1},\hdots,z_{i_d}$ directions up to order $r$ on $\Z$. Here $\phi^*H$ and $\phi^*(\tilde{H})$ are defined as follows: $$\phi^*(H)(z)=H(\phi(z)),~\text{and}~\phi^*(\tilde{H})(z)=\tilde{H}(\phi(z))$$ where $H$ and $\tilde{H}$ are Gramians \w the frames $\{s_1,\hdots,s_r\}$ and $\{\tilde{s}_1,\hdots,\tilde{s}_r\}$ normalized at $w^0$ along $\mathsf{Z}$. Since $D\phi$ is identically identity matrix the curvatures of $\phi^*E|_{\Z}\ra\Z$ (respectively, $\phi^*\tilde{E}|_{\Z}\ra\Z$) and their covariant derivatives of any order at $w\in \Z$ turn out to be those of $E\ra \O$ at $\phi(w)\in\mathsf{Z}$. In particular, we have a linear isometry from $E_{w^0}$ onto $\tilde{E}_{w^0}$ which intertwines the    curvatures $\dbar_{i_l}((\del_{i_j}H)H^{-1})$ and $\dbar_{i_l}((\del_{i_j}\tilde{H})\tilde{H}^{-1})$ of $E$ and $\tilde{E}$, respectively, for $1\leq j,l\leq d$, as well as their covariant derivatives along $z_{i_1},\hdots,z_{i_d}$ directions up to order $r$ at $w^0$. Since $w^0$ was arbitrary and $d>r$ it follows that for every $w\in\O$ there is a linear isometry $V:E_w\ra\tilde{E}_w$ which intertwines the curvatures and their covariant derivatives up to order $r$ at $w$. Therefore, from Theorem \ref{equivalence of connections} we have that $\HM$ and $\tilde{\HM}$ are locally unitarily equivalent.
\end{proof}


\begin{rem}\label{counter example}
Let $\O=\D^2$ and $K((z_1,z_2),(w_1,w_2))=(1-z_1\ov{w}_1)^{-\l}(1-z_2\ov{w}_2)^{-\mu}$ be the weighted Bergman kernel with weights $\l,\mu>0$. Denote $\H^{(\l,\mu)}(\D^2)$ be the weighted Bergman space. Consider the reproducing Hilbert modules $\H_{K_1}$ and $\H_{K_2}$ on $\D^2$ with the reproducing kernels $K_1$ and $K_2$ defined as follows:
$$K_1((z_1,z_2),(w_1,w_2))=\left(\begin{smallmatrix}
K((z_1,z_2),(w_1,w_2)) & 0\\
0 & \del_1^4\dbar_1^4\del_2^4\dbar_2^4K((z_1,z_2),(w_1,w_2))\\
\end{smallmatrix}\right),~\text{and}$$
$$K_2((z_1,z_2),(w_1,w_2))=\left(\begin{smallmatrix}
K((z_1,z_2),(w_1,w_2)) & \dbar_1^4\dbar_2^4K((z_1,z_2),(w_1,w_2))\\
\del_1^4\del_2^4K((z_1,z_2),(w_1,w_2)) & \del_1^4\dbar_1^4\del_2^4\dbar_2^4K((z_1,z_2),(w_1,w_2))\\
\end{smallmatrix}\right).$$
It can be shown that adjoint of the multiplication operators by coordinate functions on both of these Hilbert spaces are in $\mathrm B_2(\D^2)$ which are not unitarily equivalent. But since $K$ is diagonal kernel it turns out that $$((\del_1^i\dbar_1^jK_1((z_1,z_2),(w_1,w_2))))_{i,j=0}^3|_{\Z_l}=((\del_1^i\dbar_1^jK_2((z_1,z_2),(w_1,w_2))))_{i,j=0}^3|_{\Z_l}$$ where $l=1,2$ and $\Z_l$ is the coordinate plane defined by $z_l=0$. Therefore, from Theorem \ref{qm} and Proposition \ref{eococh} that quotient modules obtained from the submodules of functions in $\H_{K_1}$ and $\H_{K_2}$ vanishing of order $4$ along any connected smooth hypersurfaces in $\D^2$ are unitarily equivalent. Thus it shows that the assumption on the relation of the dimension of the domain and the rank of the bundle as mentioned in the theorem above is necessary.
\end{rem}
The author has been benefited from the discussions with Dr.  Soumitra Ghara in concluding that the reproducing kernel $K_2$ in Remark \ref{counter example} corresponds to an operator tuple in $\mathrm B_2(\D^2)$.

\vspace{0.1in}

\textbf{Acknowledgement.} The author is thankful to Professor Gadadhar Misra for his generous support and invaluable comments on this article. The author is also indebted to Dr. Shibananda Biswas for very detailed comments and suggestions in preparation of this paper. 

\end{document}